\documentclass[11pt]{amsart}
\usepackage[foot]{amsaddr}
\usepackage{amsfonts} 
\usepackage{amssymb}
\usepackage{amsthm}
\usepackage{amsmath} 
\usepackage{array}
\usepackage{multirow}
\usepackage{mathrsfs}
\usepackage{mathtools}
\usepackage{dsfont}
\usepackage{esint}
\usepackage[normalem]{ulem}
\usepackage{marginnote}
\usepackage[english]{babel} 
\usepackage{marginnote}
\usepackage{graphicx, import}
\usepackage{bm}
\usepackage{amsopn}

\usepackage[margin=1in]{geometry}
 
\usepackage{pgf,tikz}
\usetikzlibrary{arrows}
\usepackage[skip=2pt,font=scriptsize]{caption}
\usepackage{subcaption}
\usepackage{floatrow} 

\usepackage{enumerate}
\usepackage{enumitem}

\usepackage{hyperref}
\usepackage{cleveref}
\usepackage{mathtools}
\hypersetup{
    colorlinks,
    linkcolor={red!80!black},
    citecolor={green!80!black}
} 
\usepackage{amsopn}
\theoremstyle{plain}
\begingroup
\theoremstyle{plain}
\newtheorem{theorem}{Theorem}[section]
\newtheorem{corollary}[theorem]{Corollary}

\newtheorem{lemma}[theorem]{Lemma}
\newtheorem{ass}{Assumption}
\theoremstyle{definition}

\newtheorem{remark}[theorem]{Remark}

\endgroup

\theoremstyle{definition}

\numberwithin{table}{section}
\numberwithin{equation}{section}
\numberwithin{figure}{section}

\newcommand{\R}{\mathbb{R}}

\mathchardef\emptyset="001F

\newcommand{\eps}{\varepsilon}
\renewcommand{\tilde}{\widetilde}

\newcommand{\loc}{\mathrm{loc}}

\renewcommand{\Sigma}{\boldsymbol{\sigma}}

\newcommand{\x}{\boldsymbol{x}}  
\newcommand{\y}{\boldsymbol{y}} 
\newcommand{\z}{\boldsymbol{z}}

\title[Semi-autonomous NODEs]{Universal Approximation of Dynamical Systems by Semi-Autonomous Neural ODEs and Applications}

\author[Z. Li, K. Liu, L. Liverani and E. Zuazua]{Ziqian Li$^1$$^2$}
\address{$^1$School of Mathematics, Jilin University, 2699 Qianjin Street, Changchun, 130012, Jilin, China.}
\address{$^2$Chair for Dynamics, Control, Machine Learning, and Numerics (Alexander von Humboldt Professorship), Department of Mathematics, Friedrich–Alexander-Universit\"at Erlangen–N\"urnberg, 91058 Erlangen, Germany.}

\author{Kang Liu$^3$}
\address{$^3$Institut de Mathématiques de Bourgogne, Université Bourgogne Europe, CNRS, 21000 Dijon, France.}

\author{Lorenzo Liverani$^2$}

\author{Enrique Zuazua$^2$$^4$$^5$}
\address{$^4$Departamento de Matem\'aticas, Universidad Aut\'onoma de Madrid, 28049 Madrid, Spain.}
\address{$^5$Chair of Computational Mathematics, Fundaci\'on Deusto. Av. de las Universidades, 24, 48007 Bilbao, Basque Country, Spain.}
\email{zqli23@mails.jlu.edu.cn, kang.liu@u-bourgogne.fr, lorenzo.liverani@fau.de}
\email{enrique.zuazua@fau.de}

\begin{document}

\begin{abstract}
In this paper, we introduce semi-autonomous neural ordinary differential equations (SA-NODEs), a variation of the vanilla NODEs, employing fewer parameters.
We investigate the universal approximation properties of SA-NODEs for dynamical systems from both a theoretical and a numerical perspective.
Within the assumption of a finite-time horizon, under general hypotheses we establish an asymptotic approximation result, demonstrating that the error vanishes as the number of parameters goes to infinity. Under additional regularity assumptions, we further specify this convergence rate in relation to the number of parameters, utilizing quantitative approximation results in the Barron space. Based on the previous result, we prove an approximation rate for transport equations by their neural counterparts.
Our numerical experiments validate the effectiveness of SA-NODEs in capturing the dynamics of various ODE systems and transport equations. Additionally, we compare SA-NODEs with vanilla NODEs, highlighting the superior performance and reduced complexity of our approach. 
\end{abstract}


\maketitle

\section{Introduction}\label{sec:1}
\subsection{Neural ODEs} Neural ordinary differential equations (NODEs) represent a groundbreaking fusion of deep learning and differential equations \cite{chen2018neural}. This innovative approach stems from the realization that residual neural  networks \cite{He_2016_CVPR} (ResNets) can be viewed as discrete approximations of continuous dynamical systems. The traditional NODE model rules the evolution of an absolutely continuous state trajectory $\x = \x(t): [0,T] \to\mathbb{R}^d$ 
via an ordinary differential equation parameterized by a neural network,
\begin{equation}
\label{eq:trad-NODE}
\begin{dcases}
	\dot{\x} =\sum_{i = 1}^P W_i(t)\circ\Sigma(A_i(t)\x+B_i(t)). \\
    \x(0) = x_0,
 \end{dcases}
\end{equation}
Throughout the paper, we will refer to this NODE formulation as \emph{vanilla NODE}.
Here, $A_i \in \mathbb{L}^\infty([0,T];\mathbb{R}^{d \times d}), W_i \in \mathbb{L}^\infty([0,T];\mathbb{R}^d),$ and $B_i \in \mathbb{L}^\infty([0,T];\mathbb{R}^d)$ for $i = 1, \ldots, P$ are
the parameters of NODE, and $\circ$ stands for the Hadamard product. For a precise definition of the notation used in this paper, we direct the reader to Section \ref{sec:prelim}. Building on the idea of NODEs as formal limits of ResNets, the number $P$ represents the number of neurons in each ``infinitesimally thin" layer of the network parametrized by $t \in [0,T]$. The vector function $\Sigma: \mathbb{R}^d \to \mathbb{R}^d$ acts componentwise on its input as the activation function $\sigma$, which can be any of the classical activation functions such as Sigmoid, ReLU, ReLU$^k$ etc.  

NODEs are a flexible model, that can be trained to interpolate even unstructured or rough dataset, especially when these are time-dependent. However, in order to quantify the precision of the synthetic model at hand, it is often reasonable to assume that the data is simply the realization of an underlying physical law, described by a generic dynamical system of the form
\begin{equation}
\label{benchmark}
    \begin{cases}
    \dot \z = f(\z, t),\\
    \z(0) = z_0.
    \end{cases}
\end{equation}
The accuracy of the model is then assessed by measuring its deviation from the expected dynamics. ODEs systems of this form appear in a huge number of applications, for instance,  the Hamiltonian system from mechanics,  the semidiscretization of non-stationary PDEs (e.g. with the  finite elements method, see \cite[Sec.\@ 8.6.1]{allaire2007numerical} for more details), etc. Besides, the presence of a time-dependent field allows us to take external sources into account. 
For this reason, the approximation of ODE systems can be considered as a benchmark problem, and it is pivotal to develop learning architectures able to perform efficiently. This is precisely the setting of this paper.  

\subsection{Main results} 
As continuous limits of ResNets, it is natural to take the coefficients of NODEs to be time-dependent.
However, this choice entails a great increase in the complexity of the model: in practical implementations of NODEs a layer is needed for every time step, so that the number of parameters depends linearly on the number of time steps. It is then reasonable to wonder whether it is possible to decrease this complexity, while retaining the core dynamical features that play a central role in concrete applications. Furthermore, the greatest part of the existing works concerning with NODEs are interested in optimizing the coefficients $W_i(t), A_i(t)$ and $B_i(t)$ in order to drive an initial distribution of points at time $t=0$ (corresponding to the input layer) to a final target at time $t=T$ (the final layer), with little to no regards to tracking the whole trajectory over the entire interval $[0,T]$. An exception here is given by the recent work \cite{ruiz2022interpolation}. Nevertheless, it seems reasonable to expect that NODEs should be able to approximate whole trajectories, and not simply the initial and final states. Prompted by these questions, in this article we focus on a particular instance of NODEs, namely, 
\begin{equation}\label{eq:SA-NODE}
\begin{dcases}
\dot{\x} = \sum_{i=1}^P W_i \circ \Sigma(A_i^1\x+A_i^2t + B_i), \\
    \x(0) = x_0.
\end{dcases}
\end{equation}
Note that the parameters are now completely time-independent. In fact, $t$ appears only as a multiplicative factor inside of the activation function. For this reason, we dub the equation \textit{semi-autonomous NODEs} (SA-NODEs). 
This specific structural choice is not arbitrary. Indeed, it is based on the classical universal approximation result by Pinkus \cite{pinkus}, stating that every vector field $f(z,t)$, continuous on a compact set, can be approximated to arbitrary precision in the $\mathbb{L}^\infty$ norm by a shallow (single-hidden-layer) neural network of the form
\[
f_{\Theta}(z,t) = \sum_{i=1}^P W_i \circ \Sigma(A_i^1 z+A_i^2t + B_i).
\]
We refer to Theorem \ref{th:pinkus} for the full statement. This is our starting point, naturally leading to our first main result, Theorem \ref{th:UAP}, which concerns the Universal Approximation Property (UAP) for SA-NODEs. Indeed, this can be obtained by combining Pinkus Theorem with Gr\"{o}nwall estimates. However, this concatenation is far from trivial. In fact, Pinkus Theorem can only be applied on compact sets. Therefore, in order to apply Gr\"{o}nwall estimates, it is crucial to identify a suitable compact set enclosing all the SA-NODEs solutions stemming from the approximation $f_\Theta$ of the real vector field $f$.

Our main theoretical contributions, besides Theorem \ref{th:UAP}, follow a similar inspiration, arising from suitable UAP of shallow neural networks (shallow NNs). We summarize these contributions in detail below.

\begin{enumerate}
     \item The already mentioned Theorem \ref{th:UAP} establishes the UAP of SA-NODEs for the approximation of dynamical systems of the form \eqref{benchmark}. Under the sole assumption of $f$ being continuous in time and uniformly Lipschitz in space (see Assumption \ref{ass1}), we show that for any given tolerance $\eps>0$, and any compact set $K\subset\R^d$ of initial data, there exist parameters $P\geq 1$ and $W_i, A_i^1, A_i^2, B_i$ such that every trajectory of the dynamical system with initial data in $K$ is approximated in $\mathbb{L}^{\infty}(0,T)$ (up to an error of $\eps$) by the corresponding SA-NODE trajectory starting from the same initial datum. Note that this result is not concerned only with the initial and final states of the system, but with the whole trajectory, which is considered as an extension to universal approximation results provided in \cite{ruiz2023neural}.
     \item Our second result provides an upper bound on the approximation rate of SA-NODEs in relation to their width $P$, as stated in Theorem \ref{th:rates}. For this purpose, we impose the further regularity assumption that $f$ lies in the local Sobolev space $\mathcal{H}^k_{\mathrm{loc}}$ with $k > (d+1)/2 + 2$ (see Assumption \ref{ass2}). Under this setting, let $\z_{z_0}$ and $\x_{z_0}$ denote the solutions of the true dynamic \eqref{benchmark} and the SA-NODE \eqref{eq:SA-NODE}, respectively, starting from a common initial point $z_0$. Then, we establish the following error estimate:

\begin{equation}\label{intro:eq:rate}
     \sup_{(z_0,t) \in K\times[0,T]}  \,\|\z_{z_0}(t) - \x_{z_0}(t) \| \leq \frac{C_{T, K, f}}{\sqrt{P}},
\end{equation}
where $C_{T,K,f}$ is a constant independent of $P$. 
Compared to classical interpolation using the finite element method, when the vector field is smooth enough, the SA-NODE approach is free from the curse of dimensionality (see Remark \ref{rem:FEM}). 

\item 
Building on the previous result, Theorem \ref{th:Trans} establishes a universal approximation result for the transport equation \eqref{eq:trans} (with the solution denoted by \(\rho\)) using its neural counterpart \eqref{eq:neural_trans} (with the solution denoted by \(\rho_{\Theta}\)):
\begin{equation}\label{intro:eq:trans}
       \sup_{t\in [0,T]} \mathbb{W}_1(\rho(\cdot, t), \rho_{\Theta} (\cdot,t) ) \leq \frac{C_{T, f,\rho_0}}{\sqrt{P}},
    \end{equation}
    where $\rho_0$ is the initial distribution of the transport equation, $C_{T, f,\rho_0}$ is a constant independent of $P$, and $\mathbb{W}_1(\cdot,\cdot)$ is the Wasserstein-1 distance \cite[Def.\@ 6.1]{villani2009optimal}. Let us mention that this result improves the findings in \cite{ruiz2024control}, where the authors consider the approximation of the terminal time distribution \(\rho(\cdot, T)\). It also enhances the results in \cite{elamvazhuthi2022neural}, which provide a similar universal approximation result (in the \(\mathbb{W}_2\) sense) for transport equations, but lack precision in the convergence rate.

\item Finally, we present a collection of numerical results and develop a thorough performance analysis of SA-NODEs. First, we highlight the connection between our main results and the training procedure of SA-NODEs in Section \ref{sec:control} by means of classical optimal control techniques. Then, we proceed by investigating the approximation capabilities of such equations, and compare them to that of vanilla NODEs. 
We observe that SA-NODEs outperform vanilla NODEs in several respects, with the number of neurons per layer ($P$) kept fixed for a fair comparison. First, SA-NODEs involve significantly fewer parameters, resulting in reduced training time and lower storage requirements. Second, their convergence rate with respect to the number of training epochs is faster. Third, SA-NODEs achieve accurate approximation results even with smaller datasets. Finally, they exhibit superior stability in approximating both ODEs and transport equations compared to vanilla NODEs.
\end{enumerate}

\subsection{Methodology}% Main introductory sentence
Here we outline the core ideas behind the proofs of our main results. The full details are provided in Section~\ref{sec:proofs}.

% Qualitative statement
\paragraph{Qualitative convergence}
As discussed earlier, the structure of the SA-NODE is derived from the Pinkus approximation of the vector field $f$ of the original ODE. However, since the NN approximation is not uniform over the entire space, it is necessary to identify a suitable compact set that simultaneously bounds the solutions of all SA-NODEs used to construct an \( \epsilon \)-approximation of the original ODE for sufficiently small \( \epsilon \). This compact set is determined using the compactness of the initial condition set \( K \) and an \emph{a priori} bound obtained in Lemma \ref{lm:stable} via a bootstrapping argument. Replacing \( f \) with its NN approximation on this compact set and applying Grönwall’s inequality then yields the qualitative convergence of the SA-NODE solutions.

% Quantitative statement
\paragraph{Quantitative convergence}
 The convergence rate~\eqref{intro:eq:rate} arises from an \( \mathcal{O}(1/\sqrt{P}) \)-approximation of the vector field \( f \) on compact sets by shallow NNs in the \( \mathbb{L}^{\infty} \)-norm, where \( P \) denotes the number of neurons. This approximation holds for functions in the Barron space~\eqref{defn:Barron_space}; see Lemma~\ref{lm:Barron} and \cite[Thm.~2]{klusowski2018approximation}. Moreover, the Lipschitz constants of the NNs used are proved to be independent of \( P \). In Lemma~\ref{lm:Sobolev_Barron}, we show that the Sobolev space \( \mathcal{H}_{\loc}^{k} \), with \( k \ge (d+1)/2 \), embeds continuously into the Barron space. Consequently, when \( f \) lies in this Sobolev space, the desired NN approximations with uniform control over the Lipschitz constant are ensured, as stated in Corollary~\ref{cor:Barron_dd}. This uniform bound allows us to identify a suitable compact set in which all trajectories of SA-NODEs remain, enabling the application of a Grönwall-type argument to derive the estimate~\eqref{intro:eq:rate}.

% Transport equation statement
\paragraph{Transport equation}
The convergence estimate~\eqref{intro:eq:trans} for the transport equation follows by applying the bound~\eqref{intro:eq:rate} to its characteristic ODE, together with the superposition principle and the definition of the Wasserstein-1 distance. Moreover, as noted in Remark~\ref{rem:Lp}, if the vector field is approximated in the $\mathcal{W}^{1,\infty}$-norm, an analogous rate holds in the $\mathbb{L}^{p}$-norm.

\color{black}
\subsection{Related works}
NODEs fit into the more general framework of data-driven techniques for system learning and identification. With respect to other state-of-the-art paradigms, NODEs are characterized by being fully data-driven, in that they do not require the introduction of a dictionary of candidate functions (such as SINDy or methods based on Koopman operators \cite{mauroy2020koopman}), nor a priori knowledge of the physical properties of the system (such as PINNs \cite{raissi2019physics}). 
The continuous-time modeling capability of NODEs makes them particularly advantageous for applications requiring smooth interpolations and handling of irregularly sampled data, such as time series analysis \cite{rubanova2019latent} and classification \cite{ruiz2023neural}.

When information on the underlying model is available, the flexibility of NODEs allows us to tailor the structure of the differential system \eqref{eq:trad-NODE} accordingly. This is the focus of the rapidly evolving field of Structure-Preserving Learning, whose goal is to enforce desired properties into the NODE. As an example, as suggested in \cite[Section 2.2.2]{kidgerthesis}, if conservation laws driving the dynamics are known, one might employ a Hamiltonian \cite{greydanus2019hamiltonian} or Lagrangian Neural Network \cite{cranmer2020lagrangian} to build a physically meaningful right hand side in \eqref{eq:trad-NODE}. Similarly, in the recent work \cite{loya2025structurepreservingneuralordinarydifferential}, the authors enforce the longtime stability of the NODE by choosing a specific structure. Other works in this direction are \cite{celle2, Celledoni2022StructurePN}, where the authors follow the opposite approach of building a structure-preserving neural network starting from the related NODE.

From a theoretical standpoint, one of the most appealing qualities of NODEs is that their differential structure makes them suitable to be investigated by means of analysis and optimal control techniques, with the overarching goal of providing a formal justification to the behavior of classical machine learning algorithms such as ResNets. 
Several works in this direction have populated the literature in recent years. 
Concerning the controllability of such equations, we recall \cite{esteve2023sparsity, weinan2017proposal}, as well as \cite{agrachev2022control}. In these papers, an in-depth analysis was conducted concerning the capabilities of different kinds of NODEs of approximating target profiles and driving inputs to final aimpoints, both in an exact and an approximate sense. Moreover, many efforts have been devoted to uncovering the relations between the norm of the controls $W_i, A_i, B_i$ and the precision of the approximation, as well as the relation between depth and width of the NODEs \cite{alvarez2024interplay}. A property that plays a fundamental role in all of these expositions is the time-dependence of the coefficients $W_i,A_i$ and $B_i$. This effectively allows to dynamically change the region of the state space that is being affected by the NODEs, in order to move only the required inputs to the wanted targets. 

The theoretical study of NODEs extends outside the realm of controllability. Without the claim of being exhaustive, we recall the works \cite{massaroli2020dissecting, sander2022residual}, dealing with the formalization of the nature of NODEs as limits of ResNets, as well as \cite{geshkovski2022turnpike}, concerning with the long-time behavior of such equations and the dependence of their approximation properties on the final time $T$.
Another notable contribution in this field is the work by Osher et al. \cite{elamvazhuthi2022neural}, which demonstrates the UAP of the transport equation corresponding to NODEs. They show that solutions of the continuity equation can be approximated by NODEs with piecewise constant training weights to achieve an arbitrary degree of closeness.

The main technique utilized in our article relies on the universal approximation property of shallow NNs, a well-studied topic in the literature. The first result can be traced back to the Wiener Tauberian Theorem \cite[Thm.\@ II]{wiener1932tauberian} in 1932, which covers a large class of activation functions. The UAP of Sigmoidal shallow NNs was demonstrated in the celebrated work \cite{CYB} in 1989. Extensions to multilayer perceptrons were made in \cite{hornik1991approximation}. A general UAP result for non-polynomial activation functions, including ReLU, was established in \cite{leshno1993multilayer}. For a comprehensive summary of universal approximation results over the past century, see \cite{pinkus}.

Regarding quantitative results, the approximation rate in the $\mathbb{L}^2$ sense for Sigmoidal shallow NNs was investigated for functions in spectral Barron spaces in \cite{barron1993universal}. Recent work \cite{ma2022barron} extends this result to the ReLU activation function, and sharper bounds on this approximation are proved in \cite{siegel2024sharp}. For precise estimates in the high-order Sobolev sense with the ReLU$^k$ activation function, see \cite{li2024two,li2024function}. The $\mathbb{L}^\infty$-approximation rate for ReLU networks plays a central role in the proof of Theorem~\ref{th:rates}, where we rely on the result from \cite{klusowski2018approximation}. A more precise approximation rate is provided in \cite{siegel2023optimal}, which yields a sharper convergence rate as discussed in Remark~\ref{rem:approx}. We refer to \cite{devore2021neural} for a good summary of quantitative approximation results.

\subsection{Outline of the Paper}
The paper is organized as follows. The forthcoming Section \ref{sec:prelim} introduces the notation and the preliminary definitions, and states the main results, which are then proved in Section \ref{sec:proofs}. Section \ref{sec:control} is dedicated to an in-depth explanation of how SA-NODEs are trained. In the subsequent Section \ref{sec:experiments}, we present our experimental setup and results, demonstrating the efficacy of SA-NODEs in several approximation scenarios. We draw some final conclusions and discuss potential directions for future research in Section \ref{sec:conclusion}.

\section{Main results}
\label{sec:prelim}

\subsection{Notations}
Let $n,d\in \mathbb{N}_{+}$. For any $x\in \R^n$ and $p\in  \mathbb{N}_{+}$,
let $ \|x\|_{\ell^p}$ be the $\ell^p$-norm of $x$. For convenience, we denote by $\|x\|$ the Euclidean norm ($\ell^2$-norm) of $x$. 
The inner (resp.\@ Hadamard) product of $x,y\in \R^n$ is denoted by $\langle x, y\rangle $ (resp.\@ $x\circ y$),
\begin{equation*}
   \langle x, y\rangle  = \sum_{i=1}^n x_iy_i,\quad  x\circ y = (x_1 y_1,\ldots, x_n y_n).
\end{equation*}
In the sequel of this article, unless otherwise specified, we fix the activation function $\sigma$ as the ReLU function, with $\bm{\sigma}$ standing for its $d$-dimensional vector-valued form:
\begin{equation*}
    \sigma(x) = \max\{x,0\}, \quad \forall x\in \R; \quad  \bm{\sigma} (\x) = (\sigma (\x_1),\ldots, \sigma(\x_d)), \quad \forall \x \in \R^d.
\end{equation*}
Let $\Omega \subseteq \mathbb{R}^n$ be a closed set. Denote by $\mathcal{H}^k(\Omega)$ the Sobolev space \cite[Def.\@ 3.2, $p=2$]{adams2003sobolev} (for any \( k \in \mathbb{N}_{+} \)) and by $\mathcal{C}(\Omega)$ the space of continuous functions on $\Omega$, each equipped with its standard norm. For any vector-valued functions $F \in \mathcal{H}^k(\Omega;\, \mathbb{R}^d)$ and $G \in \mathcal{C}(\Omega;\, \mathbb{R}^d)$, we define their norms as
\[
\|F\|_{\mathcal{H}^k(\Omega;\, \mathbb{R}^d)} := \sqrt{\sum_{i=1}^d  \|F_i\|^2_{\mathcal{H}^k(\Omega)}}, \quad
\|G\|_{\mathcal{C}(\Omega;\, \mathbb{R}^d)} := \sup_{x\in \Omega} \|G(x)\|,
\]
where $F_i$ denotes the $i$-th component of $F$. If no confusion arises, we shall simply write \(\|F\|_{\mathcal{H}^k(\Omega)}\)  for brevity.

\subsection{Semi-Autonomous Neural ODE}
Let us consider some ODE with a vector field from $\R^{d+1}$ ($d$ dimension for space and one dimension for time) to $\R^d$. We are interested in approximating this vector field by vector-valued shallow NNs (see Corollary \ref{cor:Barron_dd}).  
This leads to the following dynamical system, which we call the Semi-Autonomous Neural ODE,
\begin{equation}
    \label{eq:semi-nODE}
    \begin{dcases}
        \dot{\x} = \sum_{i=1}^P W_i  \circ 
 \Sigma(A^1_i \x + A^2_i t + B_i), \\
        \x(0) = x_0,
    \end{dcases}
\end{equation}
where $P\in \mathbb{N}_{+}$ is the width, and $W_i\in \R^d$, $A^1_i\in \R^{d\times d}$, $A^2_i\in \R^d $, $B_i\in \R^d$, for $i=1,\ldots, P$, are the parameters of the SA-NODE. As a consequence, the number of parameters (degree of freedom, DoF) of the SA-NODE is $Pd(d+3)$.

Let $\Theta = (W_i,A^1_i,A^2_i,B_i)_{i=1}^P$. For convenience, we denote by $f_{\Theta}(x,t)$ the right-hand side (r.h.s.) of \eqref{eq:semi-nODE}.  
It is straightforward to verify that $f_{\Theta}$ is globally Lipschitz continuous with respect to $x$:
\begin{equation}\label{eq:Lipschitz}
    \| f_{\Theta}(x, t) - f_{\Theta}(y, t) \| 
    \;\leq\; L_{\Theta}\,\|x - y\|, 
    \qquad \forall\, x, y \in \mathbb{R}^d,\; \forall\, t \ge 0,
\end{equation}
where the Lipschitz constant $L_{\Theta}$ is given by
\begin{equation}\label{eq:L_constant}
    L_{\Theta}
    = 
    \left( \sum_{j=1}^d 
        \left( \sum_{i=1}^P 
            |(W_i)_j|\, \| (A^1_i)_j \| 
        \right)^2 
    \right)^{1/2}.
\end{equation}
Here, $(W_i)_j$ denotes the $j$-th component of the weight vector $W_i$, and $(A^1_i)_j$ denotes the $j$-th row of the matrix $A^1_i$.

Therefore, we deduce from the Cauchy-Lipschitz Theorem that for any parameter $\Theta$ and any initial point $\x_0$, the system \eqref{eq:semi-nODE} has a unique solution for $t\geq 0$.

\subsection{Main results}
Fix $T>0$.
Let us consider a non-autonomous ODE system with a vector field $f\colon \R^d \times [0,T] \to \R^d$ and an initial point $\z_0\in \R^d$,
    \begin{equation}
    \label{eq:ode}
    \begin{cases}
    \dot \z = f(\z, t), ~ t\in (0,T),\\
    \z(0) = z_0.
    \end{cases}
    \end{equation}
    To ensure the existence and uniqueness of the solution of \eqref{eq:ode}, we need the following assumption.

    \begin{ass}\label{ass1}
        The function $f\colon \R^d \times [0,T] \to \R^d$ is continuous in $t$ and there exists $L>0$ such that
        \begin{equation*}
            \|f(x,t) - f(y,t)\| \leq L \|x-y\|, \quad \forall (x,y)\in \R^d  \text{ and } \forall t \in [0,T].
        \end{equation*}
    \end{ass}
Our first result concerns the approximation properties of SA-NODEs.
\begin{theorem}
\label{th:UAP}
    Let Assumption \ref{ass1} hold true. For any compact set $K\subseteq \R^{d}$ and any $\eps>0$, there exists a constant $ P_{\eps, T, K, f}$ such that for any $P\geq P_{\eps, T, K, f}$, there exist parameters $(W_i,A^1_i,A^2_i,B_i) \in \R^d \times \R^{d\times d} \times \R^d \times \R^d $, for $i=1,\ldots, P$, such that 
    \begin{equation*}
      \left\|  \z_{z_0}(\cdot) - \x_{z_0}(\cdot) \right\|_{\mathcal{C}([0,T];\,\R^d)} \leq \eps, \quad \forall z_0\in K,
    \end{equation*}
    where $\z_{z_0}(\cdot)$ (resp.\@ $\x_{z_0}(\cdot)$) is the solution of \eqref{eq:ode} (resp.\@ \eqref{eq:semi-nODE}) over the time horizon $[0,T]$ with the initial state $z_0$.
\end{theorem}

We emphasize that the optimal parameters in the theorem are independent of the choice of $z_0 \in K$, which justifies referring to the process as ``learning the dynamical system" rather than merely fitting a single trajectory.

\begin{remark}\label{rem:autonomous}
    When system \eqref{eq:ode} is autonomous, Theorem \ref{th:UAP} can be recast in the exact same shape for the simpler NODE 
    \[
    \begin{dcases}
        \dot{\x} = \sum_{i=1}^P W_i  \circ 
 \Sigma(A^1_i \x + B_i), \\
        \x(0) = x_0,
    \end{dcases}
    \]
    obtained by setting $A_i^2 = 0$. Throughout the paper we have made the conscious choice of being agnostic as to whether data have been collected by an autonomous or a non-autonomous system. We believe this better reflects the nature of real-world experiments, which are often polluted by small time-dependent errors. Nevertheless, if additional knowledge on the form of \eqref{eq:ode} is available, one can adopt an autonomous NODE.
\end{remark}

\color{black}
Our second result concerns an upper bound on the approximation rate by SA-NODEs with respect to the width $P$, as stated in Theorem \ref{th:rates}. Before that, let us make an additional assumption on the regularity of the vector field $f$. 
Let $X$ be any subset of $\R^d\times [0,T]$. The local Sobolev space $\mathcal{H}_{\textnormal{loc}}^{k}(\R^d\times [0,T])$ is the set of functions such that their restriction on $X$ belongs to \( \mathcal{H}^{k}(X)\) for any compact set $X\subseteq \R^d\times [0,T]$.

\begin{ass}\label{ass2}
There exists $k>(d+1)/2+2$ such that $ 
    f \in \mathcal{H}_{\textnormal{loc}}^{k}(\R^d\times [0,T];\, \R^d)$.
\end{ass}

\begin{theorem}
\label{th:rates}
    Let Assumptions \ref{ass1}-\ref{ass2} hold true. Fix any compact set $K\subseteq \R^d$. Then, for any \(P \ge 3\), there exist parameters  $(W_i,A^1_i,A^2_i,B_i) \in \R^d \times \R^{d\times d} \times \R^d \times \R^d $, for $i=1,\ldots, P$, such that
    \begin{equation}
    \label{meanbound}
       \left\|  \z_{z_0}(\cdot) - \x_{z_0}(\cdot) \right\|_{\mathcal{C}([0,T];\,\R^d)} \leq \frac{C_{T, K, f}}{\sqrt{P}}, \quad \forall z_0\in K, 
    \end{equation}
    where $C_{T, K, f}$ is a constant independent of $P$, and $\z_{z_0}(\cdot)$ (resp.\@ $\x_{z_0}(\cdot)$) is the solution of \eqref{eq:ode} (resp.\@ \eqref{eq:semi-nODE}) over the time horizon $[0,T]$ with the initial state $z_0$.
\end{theorem}

\begin{remark}
Theorems \ref{th:UAP} and \ref{th:rates} address different aspects of the approximation properties of SA-NODEs. The former provides only a qualitative result, while the latter quantifies the precision of the approximation in terms of the number of neurons $P$. The main concession that we have to make, aside from the additional regularity required, is that the bound \eqref{meanbound} we obtain holds for any initial data in $K$. 
\end{remark}

\begin{remark}[Comparison with Finite Element Approximation]\label{rem:FEM}
   Let us compare the approximation result in Theorem~\ref{th:rates} with that obtained by interpolating the vector field \( f \) using the \( P_1 \) finite element method (FEM). Suppose Assumption~\ref{ass2} holds. By the Sobolev embedding theorem, we have
   \[
   f \in \mathcal{W}^{2,\infty}_{\text{loc}}(\mathbb{R}^{d+1}).
   \]
  Therefore, for any compact domain \( \Omega \subset \mathbb{R}^{d+1} \) with Lipschitz boundary and a regular mesh \( \Omega_h \) of mesh size \( h \), it follows from \cite[Thm.~3.1.6]{ciarlet2002finite} that there exists an approximation \( f_h \) in the corresponding finite element space such that
\[
\|f - f_h\|_{\mathbb{L}^{\infty}(\Omega)} \leq C \|f\|_{\mathcal{W}^{2,\infty}(\Omega)} \, h^2,
\]
where \( C \) depends only on the domain \( \Omega \).

    Fixing the number of basis functions \(P\), the $P1$-FEM approximation of \(f\) over a regular mesh of size \(h \sim P^{-1/(d+1)}\) yields an error of order \(\|f - f_h\|_{\mathbb{L}^{\infty}(\Omega)} = \mathcal{O}(P^{-2/(d+1)})\). 
    This complexity deteriorates rapidly with the dimension \(d\), illustrating the classical \emph{curse of dimensionality}, which persists even for highly regular functions such as \(f \in \mathcal{C}_c^{\infty}(\mathbb{R}^{d+1})\).

    In contrast, Theorem~\ref{th:rates} shows that, under Assumption~\ref{ass2}, the SA-NODE approximation achieves an \(\mathbb{L}^{\infty}\)-error of order \(\mathcal{O}(P^{-1/2})\) with respect to the number of neurons \(P\). 
    Although the prefactor associated with \(P^{-1/2}\) grows exponentially with the dimension \(d\) (see Remark~\ref{rem:constant}), the convergence rate with respect to \(P\) itself remains dimension-independent.
    Therefore, \emph{from an asymptotic point of view}, for fixed \(d\geq 4\) and large \(P\), the neural network approximation decreases faster than the classical FEM rate \(\mathcal{O}(P^{-2/(d+1)})\). 
    This indicates an asymptotic advantage of neural network–based models in high-dimensional regimes, even though the curse of dimensionality remains present in the constants.

    Finally, we note that, unlike FEM, the training of neural network parameters involves solving a non-convex optimization problem. 
    Nevertheless, in practice, these parameters can be efficiently learned using stochastic gradient descent, as discussed in Section~\ref{sec:control}.
\end{remark}

\begin{remark}[Explicit formulation of constant]\label{rem:constant}
We give an explicit bound for the constant \(C_{T,K,f}\) from Theorem~\ref{th:rates} in the setting where \(f \in \mathcal{H}^{\,d/2 + 3}_{\mathrm{loc}}\) is uniformly \(L\)-Lipschitz in the spatial variable, for some \(L>0\). Define
\[
    \mathcal{F}_{L,d}
    := \Bigl\{\, f \in \mathcal{H}^{\,d/2 + 3}_{\mathrm{loc}}(\mathbb{R}^{d+1};\mathbb{R}^d)
    \;\Big|\; f(\cdot,t)\ \text{is }L\text{-Lipschitz in }x\ \text{for all }t \Bigr\}.
\]
For simplicity, fix the initial state domain \(K=[-1,1]^d\) and a horizon \(T\ge 1\). Then, for any \(f \in \mathcal{F}_{L,d}\), the reachable set of \eqref{eq:ode}, including the time variable, is contained in
\[
    \Omega_{L,T,d} := [-T e^{LT},\, T e^{LT}]^{d+1}.
\]
There exists a constant \(C_d>0\), depending only on the dimension \(d\), such that for every \(f\in\mathcal{F}_{L,d}\) the constant \(C_{T,K,f}\) in Theorem~\ref{th:rates} satisfies
\begin{equation}\label{eq:explicit-C}
    C_{T,K,f}
    \le
    C_d\,T\,
    \bigl\| f \bigr\|_{\mathcal{H}^{\,\frac{d}{2}+3}(\Omega_{L,T,d})}\,
    \exp\Bigl(\tfrac{5}{2}LT + \sqrt{d}\,L
      + C_d\, e^{\frac{3}{2}LT}\,
        \bigl\| f \bigr\|_{\mathcal{H}^{\,\frac{d}{2}+3}(\Omega_{L,T,d})}
    \Bigr).
\end{equation}
The proof is presented in Section \ref{sec:proof_main}.
We comment on the dependence in \eqref{eq:explicit-C}:
\medskip

\begin{itemize}
    \item (Dimension dependence). 
    The factor \(C_d\) stems from the Barron-type approximation constant for functions on hypercubes \([-1,1]^d\), which does not admit a simple closed form. 
    Overall, \(C_{T,K,f}\) depends exponentially on \(d\).  Hence, a curse-of-dimensionality effect appears in the numerator of the approximation rate \eqref{intro:eq:rate}. Nevertheless, as discussed in Remark~\ref{rem:FEM}, the network error scales like \(P^{-1/2}\), whereas the classical \(P1\)-FEM error scales like \(P^{-2/(d+1)}\). Therefore, for fixed \(d\ge 4\) and large \(P\), the network approximation is asymptotically superior.

\medskip
    \item (Time dependence).
    Fixing the vector field \(f \in \mathcal{F}_{L,d}\), we observe that the constant \(C_{T,K,f}\) grows super-exponentially in time.  
    This behavior arises because the reachable domain expands exponentially with \(T\), and the approximation error of \(f\) over this domain increases accordingly with its size.
    Applying Grönwall's inequality yields an overall double-exponential growth.  
    Sharper behavior is possible when the ODE is Lyapunov stable, yielding a uniformly bounded reachable set. In practice, this blow-up can be mitigated with model-predictive control (MPC) strategies \cite{veldman2024stability}; see Remark~\ref{rem:MPC}.

\medskip
    \item (Function norm dependence).
    For a fixed horizon \(T\), the constant \(C_{T,K,f}\) depends exponentially on the Sobolev norm of $f$ in the reachable domain. This arises because the Lipschitz constant of the learned (SA-NODE) vector field scales with this norm and thus enters the Grönwall exponent. Tighter constants may be obtained by using higher-order ReLU activations to better approximate derivatives of \(f\) (see \cite[Thm.~3]{siegel2023optimal}).
\end{itemize}
\end{remark}

Applying Theorem \ref{th:rates} to the transport equation \eqref{eq:trans} associated with \eqref{eq:ode}, we obtain the third main result (in Theorem \ref{th:Trans}) on the universal approximation rate of \eqref{eq:trans} by its neural counterpart \eqref{eq:neural_trans}. The transport equation reads
\begin{equation}
\label{eq:trans}
\begin{dcases}
    \partial_t \rho + {\rm{div}}_{x}(f(x,t)\,\rho) = 0, \quad (x, t)\in \R^d\times  [0,T],\\
    \rho(\cdot,0) = \rho_0 \in \mathcal{M}(\R^d),
\end{dcases}
\end{equation}
where the main variable $\rho\colon \R^d \times \R^+ \to \R$ and $\mathcal{M}(\R^d)$ is the signed measure space. Similarly, the transport equation associated with \eqref{eq:semi-nODE}, which is the so-called neural transport equation \cite{ruiz2024control}, reads
\begin{equation}
\label{eq:neural_trans}
\begin{dcases}
    \partial_t \rho + {\rm{div}}_{x}\left(\left(\sum_{i=1}^P W_i  \circ \Sigma(A^1_i x + A^2_i t + B_i)\right)\rho\right) = 0, \quad (x, t)\in \R^d\times  [0,T],\\
    \rho(\cdot,0) = \rho_0 \in \mathcal{M}(\R^d).
\end{dcases}
\end{equation}
The connection between NODEs and transport equations is not new, and it appears naturally in the theory of normalizing flows \cite{papamakarios2021normalizing, rezende2015variational}. In particular, the approximation of the terminal time distribution  of equation \eqref{eq:trans} by \eqref{eq:neural_trans} is examined in \cite{ruiz2024control}. In the following theorem, we extend this result and achieve a uniform approximation over the time horizon. Recall the definition of the Wasserstein-1 distance for probability measures as given in \cite[Def.\@ 6.1]{villani2009optimal}.

\begin{ass}\label{ass3}
The initial datum \(\rho_0\) is a compactly supported probability measure.
\end{ass}
\begin{theorem}\label{th:Trans}
    Let Assumptions \ref{ass1}-\ref{ass3} hold true. 
    Then, for any $P\ge 3$,
    there exist parameters $\Theta = \{(W_i,A^1_i,A^2_i,B_i)\}_{i=1}^P$ such that
    \begin{equation*}
       \sup_{t\in [0,T]} \mathbb{W}_1(\rho(\cdot, t), \rho_{\Theta} (\cdot,t) ) \leq \frac{C_{T, f,\rho_0}}{\sqrt{P}},
    \end{equation*}
    where $C_{T, f,\rho_0}$ is a constant independent of $P$, $\mathbb{W}_1(\cdot,\cdot)$ is the Wasserstein-1 distance,
    and $\rho(\cdot, t)$ (resp.\@ $\rho_{\Theta}(\cdot, t)$) is the solution of \eqref{eq:trans} (resp.\@ \eqref{eq:neural_trans}) at the time  $t\in [0,T]$.
\end{theorem}

\begin{remark}[Sharper Sobolev index]\label{rem:Sobolev}
The Sobolev regularity index \((d+1)/2 + 2\) appearing in Assumption~\ref{ass2} arises from the continuous embedding of Sobolev spaces into Barron spaces, as established in Lemma~\ref{lm:Sobolev_Barron}.  
We note that a sharper version of this embedding result was proved in \cite[Thm.~1]{mao2024approximation} using techniques based on the Radon transform. By applying this refined result, the regularity requirement in Assumption~\ref{ass2} can be improved to \(k\geq (d+1)/2 + 3/2\). The conclusions of Theorems~\ref{th:rates} and~\ref{th:Trans} remain unchanged under this improvement.
\end{remark}

\begin{remark}[MPC perspective]\label{rem:MPC}
As mentioned in Remark~\ref{rem:constant}, the error exhibits a rapid theoretical blow-up over time. Hence, even if the neural network is chosen with a very large width \(P\), resulting in a small initial approximation error, this error still grows super-exponentially with \(T\). Similar growth rate holds for the transport equation case. A possible practical way to mitigate this exponential growth is to adopt a MPC perspective. Instead of training a single SA-NODE to approximate the entire time horizon \([0, T]\), we update or fine-tune the network parameters over successive, shorter time windows of length \(\tau\). This strategy serves as a compromise between SA-NODEs (\(\tau = T\)) and vanilla NODEs (\(\tau  \to 0\)). Investigating the optimal choice of the time step \(\tau \) remains an important direction for future work.
\end{remark}

\begin{remark}[Approximation in the $\mathbb{L}^p$-norm]\label{rem:Lp}
   Theorem \ref{th:Trans} provides an approximation error in the Wasserstein sense. When the initial distribution has an $\mathbb{L}^p$-density, the solution $\rho$ lies in $\mathcal{C}([0,T];\, \mathbb{L}^p(\mathbb{R}^d))$. Moreover, the approximation error can also be estimated in the $\mathbb{L}^p$ sense using the classical energy method when the vector field $f$ is approximated by a neural network in the $\mathcal{W}^{1,\infty}$-norm, i.e., the approximation controls both the function and its gradient \cite{hornik1991approximation}. In this case, the stronger approximation result in \cite[Thm.~3]{siegel2023optimal} is applicable.
\end{remark}

\section{Proof of main results}
\label{sec:proofs}

\noindent
This section is devoted to proving the main results. 

\subsection{Proof of Theorem \ref{th:UAP}}

The proof is based on the following universal approximation result due to Pinkus \cite{pinkus}, which extends the celebrated theorem of Cybenko \cite{CYB} to non-polynomial activation functions. We report it here for the reader's convenience, suitably tailored to our scopes.

\begin{theorem}[\cite{pinkus}]
\label{th:pinkus}
    Fix a compact set $X\subseteq \R^{d+1}$. Let $\sigma$ be a non-polynomial continuous function. For any function $g \in \mathcal{C}(X;\R^d)$ and $\eps > 0$ there exists parameters $(W_i,A_i,B_i) \in \R^d \times \R^{(d+1)\times d} \times \R^d$, for $i=1,\ldots, P$, such that, calling
    \[
    f_{\Theta}(x) = \sum_{i=1}^P W_i  \circ 
 \Sigma(A_i x  + B_i),\quad \forall x\in X,
    \]
    it holds
    \[
    \| g - f_{\Theta}\|_{\mathcal{C}(X;\,\R^d)}\leq \eps.
    \]
\end{theorem}

We also need the following lemma on the a priori bound of the solution of SA-NODE \eqref{eq:semi-nODE}.
\begin{lemma}[A priori bound]\label{lm:stable}
  Let Assumption \ref{ass1} hold true. For any $t\in [0,T]$, define 
\[
K_t \coloneqq \left\{x\in \R^d \, \Big|\, \|x\| \leq \sup_{z\in K} \left(\|z\| + t+ \int_{0}^t \|f(0,s)\|ds\right) \exp{(Lt)}\right\}.
\]
 Then, for any \(f_1 \in \mathcal{C}(\R^d\times[0,T];\, \R^d)\) such that $f_1$ is locally Lipschitz in $x$ and \(\|f_1 - f\|_{\mathbb{L}^{\infty}(K_T\times [0,T]; \,\R^d)} \leq 1\) and \(\y\) satisfying
\begin{equation*}
    \dot{\y} = f_1 (\y,t), \quad \y(0) = z_0 \in K,
\end{equation*}
we have \(\y(t) \in K_t\) for any \(t \in [0,T]\).
\end{lemma}
\begin{proof}
The proof follows from the standard bootstrap principle \cite[Prop.\@ 1.21]{tao2006nonlinear}.
For any $t\in [0,T]$, denote by $\mathbf{H}(t)$ the ``hypothesis”: $\|f_1(\y(s),s) - f(\y(s),s)\| \leq 1$ for any $s\in[0,t]$; and denote by $\mathbf{C}(t)$ the ``conclusion": $\y(s)\in K_s$ for any $s\in[0,t]$. First, $\mathbf{H}(0)$ is true. Then, by Gr\"{o}nwall's inequality, $\mathbf{H}(t)$ implies $\mathbf{C}(t)$. Moreover, by the assumption of $f_1$ and the definition of $K_t$, $\mathbf{C}(t)$ implies $\mathbf{H}(t')$ for $t'\in [0,T]$ in a neighborhood of $t$. Since $K_t$ is compact and continuously depends on $t$, the conclusion $\mathbf{C}(t)$ is closed. We conclude from \cite[Prop.\@ 1.21]{tao2006nonlinear}.
\end{proof}

We now prove Theorem \ref{th:UAP}. Fixing any $0< \eps <1$, we apply Theorem \ref{th:pinkus} to $f$ on $K_T$ (defined in Lemma \ref{lm:stable}), finding $P$, $W_i \in \R^{d}$, $A_i=(A_i^1, A_i^2) \in \R^{(d+1)\times d}$ and $B_i\in \R^d$, such that the function
    \[
    f_{\Theta}(x,t) = \sum_{i=1}^P W_i  \circ 
 \Sigma(A^1_i x + A^2_i t + B_i), 
    \]
approximates $f$ by $\eps$ in the $\mathbb{L}^\infty(K_T\times [0,T];\, \R^d)$ norm. Since $\epsilon <1$, by Lemma \ref{lm:stable}, we have $\x_{z_0}(t)\in K_T$ for any $t\in[0,T]$.
 Hence, recalling that $f$ is uniformly Lipschitz continuous, we have
\begin{align*}
   \|\z_{z_0}(t) - \x_{z_0}(t)\| &= \left\| z_0 + \int_0^t f(\z_{z_0}(s),s)ds - z_0 - \int_0^t f_{\Theta}(\x_{z_0}(s),s)ds\right\| \\
   &\leq \int_0^t \left\|f(\z_{z_0}(s),s) - f(\x_{z_0}(s),s) + f(\x_{z_0}(s),s) - f_{\Theta}(\x_{z_0}(s),s)\right\|ds \\
   &\leq L\int_0^t \|\z_{z_0}(s) - \x_{z_0}(s)\|ds + \eps t,
\end{align*}
for any $t\leq T$. 
Exploiting again Gr\"{o}nwall's Lemma we arrive at
\begin{equation*}
        \label{gron}
    \|\z_{z_0} - \x_{z_0}\|_{\mathbb{L}^{\infty}([0,T];\R^d)} \leq \eps T e^{L T}.
    \end{equation*}
Up to redefining $\eps$, we obtain the conclusion.

\subsection{Approximation rate in the Barron space}

Fix any compact set $X\in \R^n$ with $n\in \mathbb{N}_{+}$. Recall the definition of the Barron space on $X$ from \cite[Eq.\@ 1]{ma2022barron}:
\begin{equation}
\label{defn:Barron_space}
\begin{aligned}
    \mathcal{S}_{\text{B}}(X) \coloneqq \Big\{ f\in \mathcal{C}(X) \,\Big| \, &\exists \, \mu\in \mathcal{P}(\R^{n+2}) \\ 
    &\text{ s.t. } f(x) = \int_{\R^{n+2}} w \sigma (\langle a, x\rangle +b) d\mu(w,a,b), \, \forall x\in X \Big\},
\end{aligned}
\end{equation}
where $\mathcal{P}(\R^{n+2})$ is the set of all Borel probability measures on $\R^{n+2}$. 

Let us recall the following result, which characterizes a class of functions lying in the Barron space and establishes a uniform approximation rate by shallow NNs. This lemma is a slight refinement of \cite[Thm.~2]{klusowski2018approximation}.
\begin{lemma}\label{lm:Barron}
Let \(X = [-1,1]^n\).  Suppose \(f \in \mathcal{C}(X)\) admits an extension \(\bar f\in \mathbb{L}^1(\R^n)\) whose Fourier transform satisfies
\begin{equation}\label{eq:barron_norm}
    v_{f,2}
\;:=\;
\int_{\R^n}\|\xi\|_{\ell^1}^2\;\bigl|\mathcal F(\bar f)(\xi)\bigr|\;d\xi
\;<\;\infty.
\end{equation}
Then \(f\in \mathcal S_{\mathrm B}(X)\).  Moreover, for every integer \(P\ge3\) there exist \((w_i,a_i,b_i)\in\R^{n+2}\), for \(i=1,\dots,P\), such that
\begin{align*}
   & \Bigl\|\,f 
  \;-\;\sum_{i=1}^{P} w_i\,\sigma\bigl(\langle a_i,\,\cdot\,\rangle + b_i\bigr)
\Bigr\|_{\mathcal{C}(X)}
\;\le\;
\frac{C_n\,v_{f,2}}{\sqrt{P}}, \quad \textnormal{and}
\\
& \textnormal{Lip}\, \left( \sum_{i=1}^{P} w_i\,\sigma\bigl(\langle a_i,\,\cdot\,\rangle + b_i\bigr)\right) \leq \|\nabla f(0)\| + 2\, v_{f,2},
\end{align*}
where \(C_n>0\) depends only on the dimension \(n\). 
\end{lemma}

\begin{proof}
For any \(P\ge1\), \cite[Thm.~2]{klusowski2018approximation} provides parameters
\[
w_i\in\bigl[-2v_{f,2}/P,\;2v_{f,2}/P\bigr],\quad
\|a_i\|_1=1,\quad
b_i\in[-1,1],
\]
such that
\[
\Bigl\|\,f(x)
- \Bigl(f(0) + \langle\nabla f(0),x\rangle
  + \sum_{i=1}^P w_i\,\sigma(\langle a_i,x\rangle + b_i)\Bigr)\Bigr\|_{\mathcal{C}(X)}
\;\le\;
\frac{C_n}{\sqrt{P}},
\]
with \(C_n>0\) depending only on \(n\).  Noting that the affine term can be represented by two ReLU neurons,
\[
f(0)+\langle\nabla f(0),x\rangle
=\sigma\bigl(\langle\nabla f(0),x\rangle + f(0)\bigr)
- \sigma\bigl(\langle-\nabla f(0),x\rangle - f(0)\bigr),
\]
one obtains the claimed error for \(P\ge3\).  Finally, since \(\|a_i\|_2\le\|a_i\|_1=1\), the network’s Lipschitz constant is bounded by
\[
\|\nabla f(0)\| \;+\; \sum_{i=1}^P |w_i|\,\|a_i\|_2
\;\le\;
\|\nabla f(0)\| + 2\,v_{f,2}.
\]
The conclusion follows.
\end{proof}

The space of functions satisfying \eqref{eq:barron_norm} is  referred to as the Fourier-Lebesgue space in the literature. In the following lemma, we show that the Sobolev space $\mathcal{H}^k(X)$ (when the smoothness parameter \(k\) is sufficiently large) is continuously embedded in the Fourier-Lebesgue space, and therefore lies in \(\mathcal{S}_{\mathrm{B}}(X)\).
A sharper version of this result was established in \cite{mao2024approximation} using the Radon transform, see Remark~\ref{rem:Sobolev}.

\begin{lemma}\label{lm:Sobolev_Barron}
 Let $X=[-1,1]^n$. For any function $f\in \mathcal{H}^{k}(X)$ with $k > n/2 + 2$, we have
 \begin{equation*}
     v_{f,2} \leq  C_{n,k} \, \|f\|_{\mathcal{H}^k(X)},
 \end{equation*}
 where $v_{f,2}$ is defined in \eqref{eq:barron_norm} and $C_{n,k}>0$ depends only on \((n,k)\).
\end{lemma}

\begin{proof}
    Since $X=[-1,1]^n$ satisfies the strong local Lipschitz condition (see \cite[Def.\ 4.9]{adams2003sobolev}) and $f\in \mathcal{H}^{k}(X)$, where $k>n/2+2$, by \cite[Thm.\@ 4.12]{adams2003sobolev}, we have $f\in \mathcal{C}^{2}(X)$.
    Moreover, by \cite[Thm.\@ 5.24]{adams2003sobolev}, there exists an extension $\bar{f} \in \mathcal{H}^{k}(\R^{n})$ such that $\bar{f} |_{X} = f$. 
    Let $\mathcal{F}(\bar{f})$ denote the Fourier transform of \(\bar f\).  By the Cauchy–Schwarz inequality, 
\begin{equation*}
\begin{split}
    \int_{\R^n} \|\xi\|^2 | \mathcal{F}(\bar{f}) (\xi)| d \xi & \leq \int_{\R^n} \left(1 + \|\xi\|^2\right) | \mathcal{F}(\bar{f}) (\xi)| d \xi \\
    & \leq  \left( \int_{\R^n} \left(1 + \|\xi\|^2\right)^{2-k}  d \xi \right)^{\frac{1}{2}}  \left( \int_{\R^n} \left(1 + \|\xi\|^2\right)^k | \mathcal{F}(\bar{f}) (\xi)|^2 d \xi \right)^{\frac{1}{2}}\\
    & = \pi ^{n/4}\, \frac{\Gamma(k-2-n/2)}{\Gamma(k-2)} \, \|\bar{f}\|_{\mathcal{H}^k(\R^n)} ,
\end{split}
   \end{equation*}
where $\Gamma(\cdot)$ is the Gamma function. Since the extension operator \(E: \mathcal{H}^k(X)\to \mathcal{H}^k(\R^n)\) is bounded with norm depending only on \(n\) and \(k\), and using the inequality $\|\xi\|_{\ell^1} \leq \sqrt{n} \|\xi\|  $ for any $\xi \in \R^n$, the desired estimate follows.
\end{proof}

Recall that $\circ$ denotes the Hadamard product and that 
\(\bm\sigma\colon \R^d\to\R^d\) is the component‐wise ReLU activation. Combining the previous two lemmas yields the following corollary.

\begin{corollary}
    \label{cor:Barron_dd}
     Fix any \(m\in\mathbb{N}\) and set \(X_m=[-m,m]^n\). Let
      $F \in \mathcal{H}^{k}(X_m;\R^d)$ with $k > n/2 + 2$. 
    Then, for any $P \geq 3$, there exists $(W_i, A_i, B_i) \in \R^d \times \R^{d\times n} \times \R^{d}$, for $i=1,\ldots,P$, such that 
    \begin{align*}
       & \left\| F(\cdot) - \sum_{i=1}^P W_i \circ \bm{\sigma} (A_i \cdot + B_i) \right\|_{\mathcal{C}(X_m)} \leq \frac{C_{n,k,m} \|F\|_{\mathcal{H}^k(X_m)}}{\sqrt{P}}, \quad  \textnormal{and}\\
       &  \textnormal{Lip}\, \left( \sum_{i=1}^P W_i \circ \bm{\sigma} (A_i \cdot + B_i)\right) \leq \|\nabla F(0)\|_{\textnormal{F}} + C_{n,k,m}\,  \|F\|_{\mathcal{H}^k(X_m;\,\R^d)},
    \end{align*}  
where $C_{n,k,m}>0$ depends only on \((n,k,m)\) and $ \|\nabla F(0)\|_{\textnormal{F}}$ is the Frobenius norm of the Jacobian matrix $\nabla F(0)$.
\end{corollary}

\begin{proof}
  Fix any $i\in\{1,\ldots, n\}$.  Define the dilated function
\[
\tilde F_i(x) \;=\; F_i(m\,x),
\qquad x\in X = [-1,1]^n.
\]
We deduce that
\[
\|\tilde F_i\|_{\mathcal H^k(X)}
\;\le\;
m^{\,k-n/2}\,\|F_i\|_{\mathcal H^k(X_m)}.
\]
By Lemma \ref{lm:Sobolev_Barron}, there exists a constant $C_{n,k}>0$ such that
\begin{equation*}
    v_{\tilde{F}_i, 2} \leq C_{n,k} \, \|\tilde F_i\|_{\mathcal H^k(X)} \leq C_{n,k} \, m^{\,k-n/2}\,\|F_i\|_{\mathcal H^k(X_m)}.
\end{equation*}
Besides, by Lemma \ref{lm:Barron}, there exist \((w_j^i,a_j^i,b_j^i)\in\R^{n+2}\) for \(j=1,\dots,P\) and $C_n>0$ such that
\begin{align*}
  &  \Bigl\|\,\tilde{F}_i(\cdot) 
  \;-\;\sum_{j=1}^{P} w_j^i\,\sigma\bigl(\langle a_j^i,\,\cdot\,\rangle + b_j^i\bigr)
\Bigr\|_{\mathcal{C}(X)}
\;\le\;
\frac{C_n\,v_{\tilde{F}_i,2}}{\sqrt{P}}\; \leq \; \frac{C_n\,C_{n,k} \, m^{\,k-n/2}\,\|F_i\|_{\mathcal H^k(X_m)}}{\sqrt{P}},  \\
& \text{Lip} \left( \sum_{j=1}^{P} w_j^i\,\sigma\bigl(\langle a_j^i,\,\cdot\,\rangle + b_j^i\bigr)  \right) \leq m \|\nabla F_i(0)\| + 2 C_{n,k} \, m^{\,k-n/2}\,\|F_i\|_{\mathcal H^k(X_m)}.
\end{align*}
Recalling the definiton of $\tilde{F}_i$, we have 
\[
\Bigl\|\,F_i(\cdot) 
  \;-\;\sum_{j=1}^{P} w_j^i\,\sigma\bigl(\langle a_j^i/m,\,\cdot\,\rangle + b_j^i\bigr)
\Bigr\|_{\mathcal{C}(X_m)}
\;\le\;
\frac{ C_n\,C_{n,k}\, m^{k-n/2}\,\|F_i\|_{\mathcal H^k(X_m)}}{\sqrt{P}}.
\] 
Moreover,
\[
    \text{Lip} \left( \sum_{j=1}^{P} w_j^i\,\sigma\bigl(\langle a_j^i/m,\,\cdot\,\rangle + b_j^i\bigr)  \right) \leq \|\nabla F_i(0)\| + 2 C_{n,k} \, m^{\,k-1-n/2}\,\|F_i\|_{\mathcal H^k(X_m)}.
\]
Finally, the desired estimates come from the definition of the norms of vector-valued functions and matrices.
\end{proof}

\begin{remark}[$\mathbb{L}^{\infty}$-approximation rate]\label{rem:approx}
   Lemma~\ref{lm:Barron} and Corollary~\ref{cor:Barron_dd} establish the universal approximation rate for shallow NNs in the $\mathbb{L}^{\infty}$ norm. Our main technique is drawn from \cite{klusowski2018approximation}. We also note that comparable rates appear in \cite[Prop.~1]{bach2017breaking} and \cite[Thm.~3]{siegel2023optimal}, where, using deep tools from geometric discrepancy theory \cite{matousek1996improved}, one obtains a best rate of $P^{-1/2-3/2n}$ (in the SA-NODE case, $n=d+1$), and network’s Lipschitz constant can be uniformly bounded (independent of $P$). Consequently, the convergence rate in Theorem~\ref{th:rates} can likewise be improved to
\[
  P^{-\frac{1}{2}-\frac{3}{2(d+1)}}\,,
\]
where $d$ is the dimension of the dynamical system. 
\end{remark}

\begin{remark}[$\mathbb{L}^2$-approximation rate]\label{rem:sigma}
The $\mathbb{L}^2$-approximation rate (also of the order of $P^{-1/2}$) for ReLU networks follows directly from Hölder's inequality and the previously established $\mathbb{L}^\infty$ result.  
An alternative proof of this $\mathbb{L}^2$ rate can be obtained via Maurey’s inequality \cite[Lem.~2]{pisier1981remarques} (see also \cite{ma2022barron}).  
This method remains valid for a broader class of activation functions beyond ReLU and its powers used in the $\mathbb{L}^\infty$ setting.  
In particular, it was shown in \cite[Thm.~4]{li2020complexity} that the $P^{-1/2}$ approximation rate in the $\mathbb{L}^2$ norm holds when the activation function $\sigma$ is twice weakly differentiable and satisfies the integrability condition:
\begin{equation}\label{ass:activation}
  \int_{\mathbb{R}} \bigl|\sigma''(x)\bigr|\,(1 + |x|)\,dx < \infty.
\end{equation}
In particular, the sigmoid function meets \eqref{ass:activation}. Therefore, by a parallel argument in the next subsection, Theorem \ref{th:rates} can be reformulated to give the same \(P^{-1/2}\)–rate in the \(\mathbb{L}^2\)–error (with respect to \(z_0\)) for every \(\sigma\) satisfying \eqref{ass:activation}.
\end{remark}

\subsection{Proof of Theorem \ref{th:rates}}\label{sec:proof_main}
The proof is stated in the following two steps.

\medskip 
\noindent\textbf{Step 1} (Approximation of $f$). Under Assumption \ref{ass1}, the reachable set of \eqref{eq:ode},
\[
\Omega_T(K) = \bigl\{\z_{z_0}(t)\,\bigm|\;z_0\in K,\;t\in[0,T]\bigr\},
\]
is compact. Taking 
\begin{equation}\label{eq:m}
    m = T\, \max \,\Bigl\{\,1,\;\sup_{z_0\in K}\|z_0\|\, e^{L T}\Bigr\},
\end{equation}
by Grönwall’s lemma, we have 
\[
X_m \coloneqq [-m,m]^{d+1}\;\supseteq\;\Omega_T(K)\times[0,T].
\]
By Assumption \ref{ass2}, 
\[
f |_{X_m} \in \mathcal{H}^k(X_m;\, \R^d),\quad \text{ with }\, k>(d+1)/2+2.
\]
Therefore, according to Corollary \ref{cor:Barron_dd}, for any $P \geq 3$, there exists parameter $\Theta = (W_i,A^1_i,A^2_i,B_i)_{i=1}^P$ such that 
 \begin{align}
        &\left\| f(\cdot,\cdot) - f_{\Theta} (\cdot,\cdot) \right\|_{\mathcal{C}(X_m;\, \R^{d})} \leq \frac{C_{d,k,m} \|f\|_{\mathcal{H}^k(X_m)}}{\sqrt{P}},  \label{eq:approx_barron} \\
        & \text{Lip} \left(f_{\Theta} (\cdot,\cdot)\right) \leq \|\nabla f(0,0)\|_{\text{F}} +C_{d,k,m} \|f\|_{\mathcal{H}^k(X_m)}, \label{eq:lip}
    \end{align}
   where $C_{d,k,m}>0$ depends only on \((d,k,m)\).

\medskip
\noindent\textbf{Step 2} (Decomposition and estimates of the error).
For any $(z_0, t) \in K\times [0,T]$, by the triangle inequality, 
\begin{align*}
 & \| \z_{z_0}(t) - \x_{z_0}(t) \|\\
   = &\left\| \int_0^t f(\z_{z_0}(s),s) - f_{\Theta}(\z_{z_0}(s),s) + f_{\Theta}(\z_{z_0}(s),s) - f_{\Theta}(\x_{z_0}(s),s) \,ds \right\|\\
   \leq & \underbrace{\int_{0}^t \|f(\z_{z_0}(s),s) - f_{\Theta}(\z_{z_0}(s),s)\| \,ds}_{=\colon\gamma_1} + \underbrace{\int_{0}^t \|f_{\Theta}(\z_{z_0}(s),s) - f_{\Theta}(\x_{z_0}(s),s)\| \,ds}_{=\colon\gamma_2}.
\end{align*}
Since \(\z_{z_0}(s)\in\Omega_T\) for all \(s\in[0,T]\), it follows that \((\z_{z_0}(s),s)\in X_m\). Hence, by \eqref{eq:approx_barron}, for any \(t\in[0,T]\) we obtain
\[
\gamma_1 \;\le\; \underbrace{C_{d,k,m}\,\|f\|_{\mathcal{H}^k(X_m)}}_{=\colon C_1}\;\frac{t}{\sqrt{P}} .
\]
On the other hand, by \eqref{eq:lip}, we have
\[
\gamma_2 \leq \underbrace{\left( \|\nabla f(0,0)\|_{\text{F}} +C_{d,k,m} \|f\|_{\mathcal{H}^k(X_m)} \right)}_{=\colon C_2} \int_{0}^t \|\z_{z_0}(s)-\x_{z_0}(s)\| \,ds .
\]
Here, the constants \(C_1\) and \(C_2\) depend only on \(T\), \(K\), and \(f\), since \(d\) is the dimension of the state variable of \(f\), and \(m\) is an explicit function of \(T\), \(K\), and the Lipschitz constant of \(f\), as defined in \eqref{eq:m}.
Combining the three preceding inequalities yields, for all \((z_0,t)\in K\times [0,T]\),
\[
\bigl\|\z_{z_0}(t)-\x_{z_0}(t)\bigr\|
\;\le\;
C_1\,\frac{t}{\sqrt{P}}
\;+\;
C_2\int_{0}^{t} \bigl\|\z_{z_0}(s)-\x_{z_0}(s)\bigr\|\,\mathrm{d}s,
\]
Applying Gr\"{o}nwall's lemma to the previous inequality, we deduce that for any $z_0\in K$,
\begin{equation*}
    \sup_{t\in [0,T]} \bigl\|\z_{z_0}(t)-\x_{z_0}(t)\bigr\| \leq \frac{ T C_1 e^{C_2T}}{\sqrt{P}}.
    \end{equation*}
The conclusion of Theorem \ref{th:rates} follows.

\medskip
For completeness, we provide the proof of the explicit constant stated in Remark~\ref{rem:constant}.
\begin{proof}[Proof of Remark~\ref{rem:constant}]
In the setting of Remark~\ref{rem:constant}, the constant \(m\) appearing in the previous proof can be specified as
\[
    m = T e^{L T}.
\]
Hence, the corresponding domain is
\[
    \Omega_{L,T,d} = X_m = [-m, m]^{d+1}.
\]

From the proof of Corollary~\ref{cor:Barron_dd}, and using that 
\( f \in \mathcal{H}^{\,d/2 + 3}_{\mathrm{loc}}(\mathbb{R}^{d+1}; \mathbb{R}^d) \) (so \(k = d/2 + 3\)), 
we can make the constants in estimates~\eqref{eq:approx_barron}--\eqref{eq:lip} explicit:
\begin{align*}
    &\| f - f_{\Theta} \|_{\mathcal{C}(X_m;\, \mathbb{R}^d)}
        \le \frac{C_d\, m^{5/2}\, \| f \|_{\mathcal{H}^{\,d/2 + 3}(\Omega_{L,T,d})}}{\sqrt{P}}, \\[4pt]
    &\mathrm{Lip} \big(f_{\Theta}\big)
        \le \|\nabla f(0,0)\|_{\mathrm{F}}
          + C_d\, m^{3/2}\, \| f \|_{\mathcal{H}^{\,d/2 + 3}(\Omega_{L,T,d})},
\end{align*}
where \(C_d > 0\) is a universal constant depending only on the dimension \(d\), arising from the product of the constants in Lemma~\ref{lm:Barron} and Lemma~\ref{lm:Sobolev_Barron} (with \(n = d + 1\) and \(k = d/2 + 3\)).

Since \(f\) is \(L\)-Lipschitz in space, we have
\[
    \|\nabla f(0,0)\|_{\mathrm{F}} \le \sqrt{d}\,L.
\]
Combining these estimates gives the following explicit forms of the constants \(C_1\) and \(C_2\) from the previous proof:
\begin{align*}
    C_1 &= C_d\, m^{5/2}\, \| f \|_{\mathcal{H}^{\,d/2 + 3}(\Omega_{L,T,d})}
          = C_d\, e^{\frac{5LT}{2}}\, \| f \|_{\mathcal{H}^{\,d/2 + 3}(\Omega_{L,T,d})}, \\[4pt]
    C_2 &\le \sqrt{d}\,L
          + C_d\, m^{3/2}\, \| f \|_{\mathcal{H}^{\,d/2 + 3}(\Omega_{L,T,d})}
          = \sqrt{d}\,L + C_d\, e^{\frac{3LT}{2}}\, \| f \|_{\mathcal{H}^{\,d/2 + 3}(\Omega_{L,T,d})}.
\end{align*}
Therefore, the constant \(C_{T,K,f}\) satisfies
\begin{equation*}
    \begin{split}
         C_{T,K,f}
    &= T C_1 e^{C_2 T}\\
    &\le C_d\, T\,
       \| f \|_{\mathcal{H}^{\,d/2 + 3}(\Omega_{L,T,d})}\,
       \exp\Bigl(
           \tfrac{5}{2} L T
           + \sqrt{d}\,L
           + C_d\, e^{\frac{3LT}{2}}\, \| f \|_{\mathcal{H}^{\,d/2 + 3}(\Omega_{L,T,d})}
       \Bigr),
    \end{split}
\end{equation*}
which gives the desired explicit bound~\eqref{eq:explicit-C}.
\end{proof}

\subsection{Proof of Theorem \ref{th:Trans}}
By Assumption \ref{ass1} and the fact that $\sigma$ is the ReLU function, we have
\begin{equation*}
    f,f_{\Theta} \in \mathbb{L}^1\left([0,T]; \mathcal{W}^{1,\infty}_{\text{loc}}(\R^d;\R^d)\right), \quad \text{and} \quad  \frac{\|f\|}{1+\|x\|},\frac{\|f_{\Theta}\|}{1+\|x\|} \in \mathbb{L}^1\left([0,T]; \mathbb{L}^{\infty}(\R^d)\right),
\end{equation*}
where $\mathcal{W}^{1,\infty}_{\text{loc}}$ is the local Sobolev space. By \cite[Prop.\@ 4 and Rem.\@ 7]{ambrosio2014continuity}, we have the following representations of the solutions of \eqref{eq:trans} and \eqref{eq:neural_trans}:
\begin{equation}\label{eq:represent}
    \rho(\cdot, t) = \phi_t \# \rho_0, \quad   \rho_{\Theta}(\cdot, t) = \phi_{\Theta, t} \# \rho_0, \quad \forall t\in [0,T],
\end{equation}
where $\#$ is the push-forward operator, $ \phi_t $ (resp.\@ $\phi_{\Theta,t}$) is the mapping from the initial state to the solution of \eqref{eq:ode} (resp.\@ \eqref{eq:semi-nODE}) at the time $t$. Therefore, $\rho(\cdot, t), \rho_{\Theta}(\cdot, t) \in \mathcal{P}(\R^d)$, and they are supported in a compact set by Gr\"{o}nwall's inequality (since $\text{supp}(\rho_0)$ is compact). Therefore, $\mathbb{W}_1(\rho(\cdot, t), \rho_{\Theta}(\cdot, t))$ can be calculated by \cite[Eq.\@ 6.3]{villani2009optimal}:
\begin{equation*}
    \mathbb{W}_1(\rho(\cdot, t), \rho_{\Theta}(\cdot, t)) = \sup_{\text{Lip}(g)\leq 1 } \int_{\R^d} g(x)\, d \left(\rho (x, t) - \rho_{\Theta}(x, t)\right).
\end{equation*}
Let $K$ denote the support set of $\rho_0$. By \eqref{eq:represent}, we have
\begin{equation*}
\begin{split}
    \mathbb{W}_1(\rho(\cdot, t), \rho_{\Theta}(\cdot, t)) & = \sup_{\text{Lip}(g)\leq 1 } \int_{K} g( \phi_t(z) ) -  g( \phi_{\Theta, t}(z) ) \,d \rho_0(z) \\
    &\leq \int_{K} \| \phi_t(z) - \phi_{\Theta, t}(z) \| \,d \rho_0(z).
\end{split}
\end{equation*}
For any $z\in K$, by Theorem \ref{th:rates}, there exists $C_{T,K,f}$ such that for any $z\in K$,
\[ \| \phi_t(z) - \phi_{\Theta, t}(z) \| \leq \frac{C_{T,K,f}}{\sqrt{P}}. \]
The conclusion follows.

\section{Training Strategy for SA-NODEs}
\label{sec:control}

This section is devoted to formulating optimization problems for training the parameters of SA-NODEs \eqref{eq:semi-nODE} to approximate a given dynamical system in the time horizon $[0,T]$ with initial points in a compact set $K$:
\begin{equation*}
    \begin{cases}
        \dot \z_{z_0} = f(\z_{z_0}, t), \quad & t\in [0,T], \\
        \z_{z_0}(0) = \z_0, \quad & z_0\in K.
    \end{cases}
\end{equation*}

The most straightforward setting arises when the vector field \(f(x,t)\) is known at some spatial locations and time samples. In such cases, a direct interpolation of \(f\) using a shallow NN is feasible and yields an SA-NODE.

However, in practice, we typically do not have direct access to samples of the vector field. Instead, the more commonly available observations are the positions of sensors moving along the flow \(\z_{z_0}(t)\), generated by the dynamical system with distinct initializations \(z_0\).

From this perspective, the training task becomes an inverse problem, where the goal is to infer the underlying dynamics from observed sensor trajectories. 
We begin by considering the training problem in the continuous data setting (infinite sensors and continuous time) for simplicity of presentation. This setting naturally leads to an optimal control formulation, given in \eqref{pb:OC}. In Theorem~\ref{th:grad}, we derive the gradient of the objective functional using an adjoint variable, which plays a central role in the implementation of gradient-based optimization methods. A discretized version of the optimal control problem, appropriate for finite training datasets, is presented in \eqref{eq:loss_dis}. Furthermore, a similar training framework can be extended to transport equations, as discussed in Remark~\ref{remark:trans}.

To determine the optimal parameter \(\Theta = (W, A^1, A^2, B)\) for SA-NODEs \eqref{eq:semi-nODE} in the continuous-data regime, we consider the following optimal control problem:

\begin{equation}
\label{pb:OC}
\begin{aligned}
\inf_{\Theta}\quad & L(\Theta) = \int_{0}^{T}\!\!\int_{K} \bigl\|\z_{\z_0}(t) - \x_{\z_0}(t)\bigr\|^2 \,\mathrm{d}\z_0\,\mathrm{d}t
  + \lambda \,g(\Theta), \\[0.6em]
\text{s.t.}\quad & \dot{\x}_{\z_0}(t) = f_{\Theta}\bigl(\x_{\z_0}(t),t\bigr), 
  \; \x_{\z_0}(0) = \z_0,\;\forall\,\z_0\in K.
\end{aligned}
\end{equation}
\color{black}
where \(g\) denotes a general regularization term, preceded by a positive coefficient \(\lambda\), and $f_{\Theta}$ is the vector field of \eqref{eq:semi-nODE}. Even though the approximation rate is established in the $\mathbb{L}^{\infty}$-norm, we use the $\mathbb{L}^2$-residual as the fidelity term. This choice is standard in regression tasks and is more amenable to the gradient descent algorithm.

For the choice of \(g\), we propose several options. First, the \(\ell^p\)-norm of \(\Theta\) is a classical choice in supervised learning. Second, the Lipschitz constant \eqref{eq:L_constant} of SA-NODE is effective for promoting generalization in a distributional sense, see \cite[Sec.~3]{liu2025representation} for related discussion. Third, other norms associated with shallow NNs may also be used, such as the extended Barron norm, the variation norm, and the Radon–BV seminorm, see \cite{li2024function} for a discussion of their equivalence.

\medskip

Considering $x_{z_0}$ as an implicit function of $\Theta$, by the classical adjoint method \cite[p.\@ 261-265]{luenberger1997optimization}, we obtain the gradient of the loss function $L$ in the following theorem.

\begin{theorem}\label{th:grad} For any $(\Theta,x,t)\in \R^{2Pd(d+1)} \times \R^d\times [0,T]$, let $\tilde f (\Theta, x, t) = f_{\Theta}(x,t)$. Assume that $g$ is locally Lipschitz continuous. It holds that
    \[
    \nabla L(\Theta) = \int_0^T \int_{K}  \frac{\partial \tilde{f}}{\partial\Theta} (\Theta, \x_{z_0}(t),t )^{\top}\, \boldsymbol{a}_{z_0} 
 (t)  dz_0 dt + \lambda \,\nabla g(\Theta), \quad \text{for } \Theta \text{ a.e.,}
    \]
    where $\x_{z_0}$ satisfies the SA-NODE \eqref{eq:semi-nODE} and $\boldsymbol{a}_{z_0}$ satisfies the adjoint equation
    \[
    \begin{cases}
        -\dot{\boldsymbol{a}}_{z_0}(t) =   \frac{\partial\tilde{f}}{\partial x} (\Theta, \x_{z_0}(t),t)^{\top} {\boldsymbol{a}}_{z_0} (t)+ 2(\x_{z_0}(t) - \z_{z_0}(t)  ), \quad& t\in [0,T], \\[0.6em]
        \boldsymbol{a}_{z_0} (T)  =   0,& z_0\in K.
    \end{cases}
    \]
\end{theorem}
We omit the proof, which is a consequence of \cite[Prop.\@ 1, p.\@ 262]{luenberger1997optimization}. A similar result is proved for fixed $z_0$ in \cite[Thm.\@ 1]{massaroli2020dissecting}. This theorem delineates the general procedure employed to train an SA-NODE, which consists in optimizing the coefficients via the gradient descent algorithm, where the gradient is computed by solving the adjoint equation.

\begin{remark}
    In our case, the activation function \(\sigma\) is ReLU. Consequently, function \(\tilde{f}\) in Theorem \ref{th:grad} is locally Lipschitz continuous, and thus is differentiable with respect to \(\Theta\) and \(x\) almost everywhere. This implies that the representation formula of \(\nabla L\) holds for \(\Theta\) almost everywhere.
In the adjoint equation, for any fixed \(\Theta\), the Lipschitz continuity of \(\tilde{f}\) with respect to \(x\) ensures that the vector field has a uniformly bounded divergence on \(\boldsymbol{a}_{z_0}\). This implies the well-posedness of the adjoint equation. 
\end{remark}

 Finally, since in concrete applications it is not possible to deal with a continuum of points, we ought to discretize the integrals appearing in the loss function. To this end, assume the training dataset has the structure $\{\z_k(t_l) \},k=1,2,\cdots, N,\ l=1,2,\cdots,M$, where $\z_k$ is the $k$-th trajectory among $N$ trajectories (with $N$ initial positions) and $t_l$ refers to the $l$-th step of $M$ total time steps. Then we obtain the finite-dimensional counterpart of \eqref{pb:OC}:
\begin{equation}
    \hat{L}(\Theta) = \frac{1}{NM}\sum_{k=1}^{N} \sum_{l=1}^{M} \left( \z_k(t_l) - \x_k(t_l, \Theta) \right)^2 + \lambda \,g(\Theta).
\label{eq:loss_dis}
\end{equation}
Here, $\x_k(t_l; \Theta)$ is the model's prediction at the time $ t_l $ of trajectory $k$. The gradient of \(\hat{L}\) can be computed similarly to Theorem \ref{th:grad} in this discrete context, with the backpropagation algorithm fulfilling the role of the adjoint equation.

\medskip
For the training of the transport equation, we employ the following remark to recover the ODE training strategy.
\begin{remark}[Training strategy for transport equations]
    \label{remark:trans}
   To train the parameters in the neural transport equation \eqref{eq:neural_trans} for approximating the original PDE \eqref{eq:trans},  
    we consider the corresponding characteristic system associated with \eqref{eq:neural_trans}, given by
\begin{equation}
\label{eq:neural_trans_charac}
\begin{dcases}
   % \frac{dt}{ds} = 1, \\
    \frac{d\x}{dt} = \sum_{i=1}^P W_i  \circ \Sigma(A^1_i \x + A^2_i t + B_i), \\
    \frac{d\rho}{dt} = -{\rm{div}}_{x} \left(\sum_{i=1}^P W_i \circ \Sigma(A^1_i \x + A^2_i t + B_i)\right) \rho,
\end{dcases}
\end{equation}
where $\rho$ is the density of the flow along the trajectory \(\x(t)\).
Since the activation function \(\Sigma\) is known explicitly, the second equation is equivalent to
\begin{equation*}
\label{eq:oderho}
\frac{d\rho}{dt} = -\rho \left( \sum_{i=1}^P \left\langle W_i ,\, \mathrm{diag}(A^1_{i})\Sigma'\left(A^1_i \x + A^2_i t + B_i\right) \right\rangle \right),
\end{equation*}
where $\mathrm{diag}(A_i^1)$ is the diagonal part of $A^1_i$. Indeed, we can recover the parameters by applying our ODE framework to the first line of \eqref{eq:neural_trans_charac}, which governs the trajectory positions and yields the same loss function as in \eqref{eq:loss_dis}.
However, in the transport setting, we also have access to additional information on the density along the trajectories, provided by the second line of \eqref{eq:neural_trans_charac}.
Incorporating a residual error term for the density into the loss function has the beneficial side effect of enhancing the generalization performance of the SA-NODE. This enriched loss formulation is the one we adopt in the numerical experiments.
\end{remark}

\section{Numerical Experiments}
\label{sec:experiments}

\noindent
In this section, we present several numerical results to demonstrate the capability of SA-NODEs in accurately simulating both ODEs and transport equations. Additionally, we conduct experiments to compare the performance of SA-NODEs {\eqref{eq:semi-nODE}}, with that of vanilla NODEs {\eqref{eq:trad-NODE}}, providing evidence for the superior effectiveness and precision of SA-NODEs in these contexts. The implementation of the code is carried out in Python using the PyTorch library for deep learning. All experiments were performed on a workstation with two 24-core Intel Xeon Platinum 8269CY CPUs, one Nvidia RTX A6000 GPU, 512GB RAM, and an Ubuntu 20.04 operating system that implements PyTorch. The codes for all examples are publicly available from the GitHub repository \url{https://github.com/DCN-FAU-AvH/SA-NODEs}.
 
\subsection{Simulations of ODEs}
The dataset used for training and evaluation consists of batches of trajectories, computed from the exact system using the fourth-order Runge-Kutta method over the time interval $[0, 5]$ with a time step of $0.05$. The initial conditions are sampled from a grid with coordinates ranging in $[-2, 2]$ in increments of $0.2$ in both $z_1$ and $z_2$ dimensions. This results in a total of $441$ trajectories, with only half of them randomly chosen to be utilized for training (i.e.\@ 220 trajectories). Note that in the forthcoming pictures we have limited ourselves to plotting only $100$ trajectories for clarity. We demonstrate that even with this relatively limited amount of data, the SA-NODE is capable of capturing the underlying dynamical system. In the following figures, red lines represent the simulated results of the training dataset by NODEs, while green lines represent the simulated results of the testing dataset by NODEs. These green indicators are crucial for assessing the model's generalization capability and how well it can predict the dynamics of unseen initial data. The neural network consists of 1000 neurons in the hidden layer and ReLU as the activation function. For training, we use the Adam optimizer with an initial learning rate of $10^{-3}$, decaying it by a factor of $0.8$ every $1000$ epochs over a total of $10000$ epochs. The weight parameter \(\lambda\) in the loss function~\eqref{eq:loss_dis} is set to \(10^{-4}\), and the regularization function \(g\) is defined using the Lipschitz constant in~\eqref{eq:L_constant}.

Figure \ref{fig:ODEs} summarizes our findings: on the left, we plot the evolution simulated by the SA-NODEs; in the center, the solution to the exact system; and on the right, the mean and standard deviation of errors. Here, the error for trajectory $k$ is defined by $e_k(t) = \|\z_{k}(t) - \x_k(t)\|$. In the right part of Figure \ref{fig:ODEs}, the red (resp.\@ blue) curve represents the mean value of \(e_k\) in the training (resp.\@ testing) set, while the shaded gray bounds indicate the standard deviation of \(e_k\) in the testing set.

\bigskip
\noindent
\textbf{Example 1: Nonlinear Autonomous ODEs}
\medskip

\noindent Nonlinear ODEs present a great challenge due to the complexity and variety of behaviors they exhibit. Unlike linear systems, which have well-understood and predictable solutions, nonlinear systems can show phenomena such as limit cycles, chaos, and bifurcations, making them harder to analyze and approximate. The nonlinear ODE system example is the undamped pendulum, which is described by
\begin{equation}
\label{eq:exam_ODE_auto}
\begin{cases}
    \dot{z}_1 = z_2, \\
    \dot{z}_2 = -\sin(z_1).
\end{cases}
\end{equation}
As shown in Figures \ref{fig:ODE_auto} and \ref{fig:ODE_auto_error}, the SA-NODE captures the behavior of the underlying dynamical system, albeit with a gradual reduction in accuracy over longer time horizons. We conjecture that this is due to the dual nature of this system, which presents periodic trajectories or unbounded trajectories depending on the initial conditions. We also note that the bad performance is mostly concentrated on the testing dataset, meaning that the SA-NODE retains good simulation properties even for this complex system.

\bigskip
\noindent
\textbf{Example 2: Nonlinear Non-Autonomous System}
\medskip

\noindent Nonlinear non-autonomous ODEs can model complex phenomena such as forced oscillations in mechanical systems and varying environmental influences in biological systems. Solving these kinds of ODEs is challenging due to the intricate interplay between nonlinearity and time-dependence, leading to phenomena like bifurcations, chaos, and sensitivity to initial conditions. We consider the following nonlinear non-autonomous ODE system
\begin{equation}
\label{eq:exam_ODE_nonauto}
\begin{cases}
    \dot{z}_1 = z_2, \\
    \dot{z}_2 = z_1 -z_1^3+\delta \cos(\omega t).
\end{cases}
\end{equation}
This is known as the forced Duffing equation, and it is used to model certain damped and driven oscillators, where $\delta$ controls the amount of damping and $\omega$ is the angular frequency of the periodic driving force. In the following experiments, $\delta=0.1$ and $\omega=\pi$. Figure \ref{fig:ODE_nonauto} shows SA-NODEs simulates well with the nonlinear non-autonomous system and Figure \ref{fig:ODE_nonauto_error} further demonstrates the high accuracy.

\begin{figure}
\centering

\begin{subfigure}[t]{0.62\textwidth}
\includegraphics[width=\textwidth]{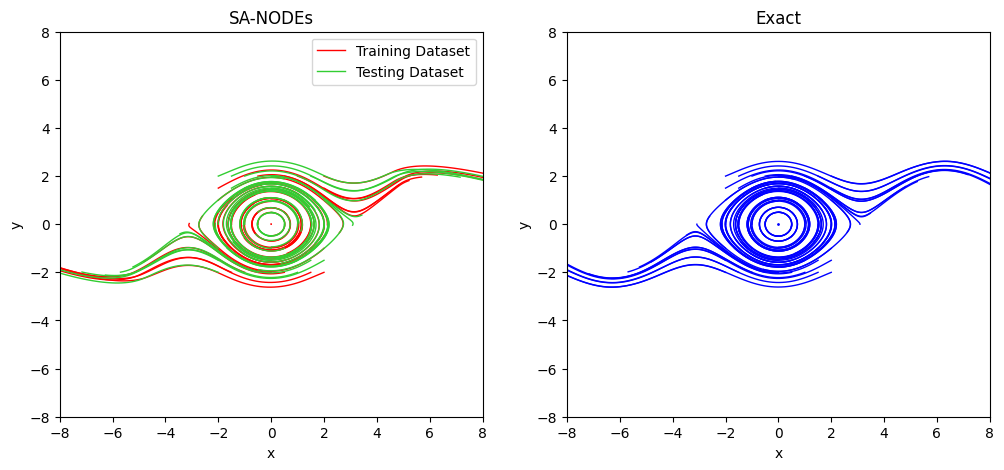}
\caption{SA-NODEs and exact solution of system \eqref{eq:exam_ODE_auto}.} \label{fig:ODE_auto}
\end{subfigure}\hspace*{\fill}
\begin{subfigure}[t]{0.32\textwidth}
\includegraphics[width=\textwidth]{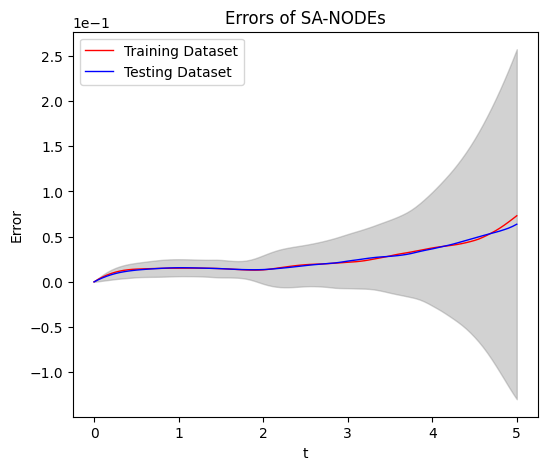}
\caption{Errors of system \eqref{eq:exam_ODE_auto}.} \label{fig:ODE_auto_error}
\end{subfigure}

\begin{subfigure}[t]{0.62\textwidth}
\includegraphics[width=\textwidth]{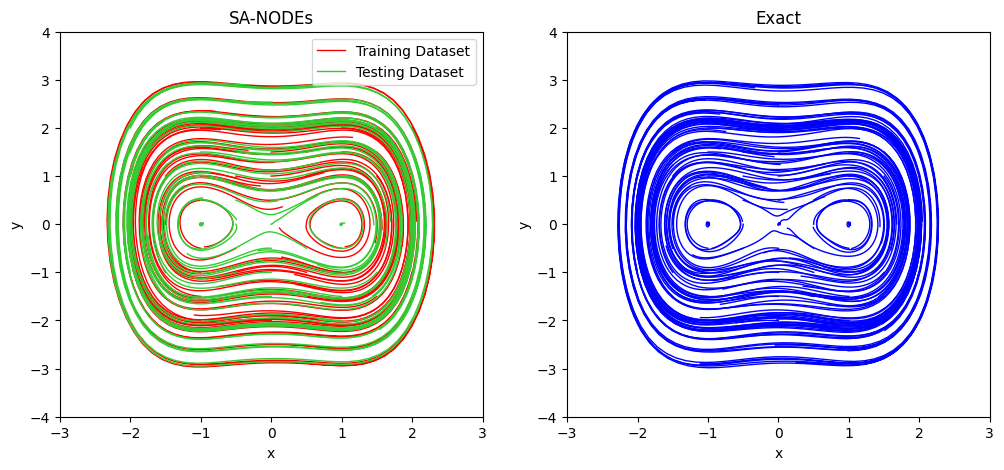}
\caption{SA-NODEs and exact solution of system \eqref{eq:exam_ODE_nonauto}.} \label{fig:ODE_nonauto}
\end{subfigure}\hspace*{\fill}
\begin{subfigure}[t]{0.32\textwidth}
\includegraphics[width=\textwidth]{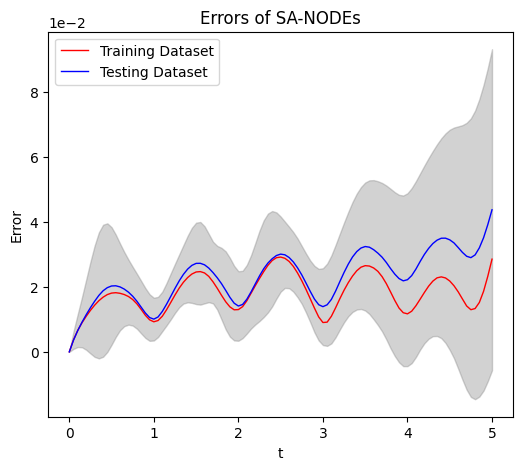}
\caption{Errors of system \eqref{eq:exam_ODE_nonauto}.} \label{fig:ODE_nonauto_error}
\end{subfigure}

\caption{SA-NODEs solution, exact solution and errors of ODE systems.} \label{fig:ODEs}
\end{figure}

%%%%%%%%%%%%%%%%%%%%%%%%%%%%%%%%%%%%
\subsection{Comparison with Vanilla NODEs}
\label{sec:com}
\noindent
In this subsection, we compare the approximation performance of vanilla NODEs \eqref{eq:trad-NODE} and SA-NODEs \eqref{eq:semi-nODE}. The comparison will focus on two primary metrics: the accuracy of the models, measured by their errors, and the complexity of the models, quantified by the number of parameters required in the neural network. To ensure a fair comparison, we trained each model for an identical, sufficiently large number of epochs ($10^4$) and used the same learning rate ($10^{-3}$). The only bottleneck was the number of employed neurons $P$. Furthermore, both methods use the $\ell^1$ norm of all NODE parameters as the regularization term in their loss function.

We first present numerical results for the autonomous system \eqref{eq:exam_ODE_auto} and the non-autonomous system \eqref{eq:exam_ODE_nonauto} in Figures \ref{fig:com_autoODE} and \ref{fig:com_nonautoODE}, respectively. Figure \ref{fig:com_autoODE_sol} and \ref{fig:com_nonautoODE_sol} compare the solutions obtained by vanilla NODEs, SA-NODEs, and the exact solution, along with the evolution of testing errors in Figure \ref{fig:com_linearODE_error} and \ref{fig:com_nonautoODE_error}. We observe that SA-NODEs demonstrate better approximation performance in terms of both accuracy and smoothness.

To provide further comparison results, we present in Table \ref{table:comp} the errors and degrees of freedom (DoF) for NODEs with different sizes. Here, \(e_{\text{max}}\) represents the maximum value of the mean error in the testing set, while \(e_T\) represents the terminal value. Recall that $P$ is the number of neurons in each hidden layer, $M$ is the number of time steps, and $d$ is the dimension of the problem. The DoF of the vanilla NODEs is $(d+3)dMP$, while the DoF of the SA-NODEs is $(d+3)dP$. Observing that the number of parameters of SA-NODEs is independent of $M$, this leads to a significant reduction in complexity when $M$ is large.

From Table \ref{table:comp}, we observe that for a fixed \(P\), the error of SA-NODEs is consistently smaller than that of vanilla NODEs, along with a significant reduction in DoF. Additionally, as \(P\) increases, the errors decrease, which is consistent with Theorem \ref{th:rates}.

Additionally, we evaluate the approximation performance of vanilla NODEs and SA‑NODEs under varying numbers of training epochs and dataset sizes. In the left panel of Figure \ref{fig:com_nonautoODE_epoch_data}, we plot the maximum mean error of both models as the number of epochs increases, showing that SA‑NODEs converge significantly faster than vanilla NODEs. In the right panel, we plot the maximum mean error against the training-set size (number of trajectories). We vary the mesh size $\Delta x\in\{1.0,0.5,0.4,0.2,0.1\}$, which corresponds to $12,\,40,\,60,\,220$, and $840$ trajectories, respectively. The results show that SA-NODEs achieve convergence with far fewer trajectories than vanilla NODEs. We summarize the comparison as follows:
\begin{enumerate}
\item With a fixed training dataset size, the training of SA-NODEs converges significantly faster than the one of vanilla NODEs, resulting in reduced computational cost.
\item For small size training datasets, SA-NODEs consistently outperform vanilla NODEs, offering a clear advantage in data-scarce regimes.
\item When both the training dataset and the number of training epochs are sufficiently large, the two models exhibit comparable performance.
\end{enumerate}

\begin{figure}[h!]
\centering
\begin{subfigure}[t]{0.72\textwidth}
\includegraphics[width=\linewidth]{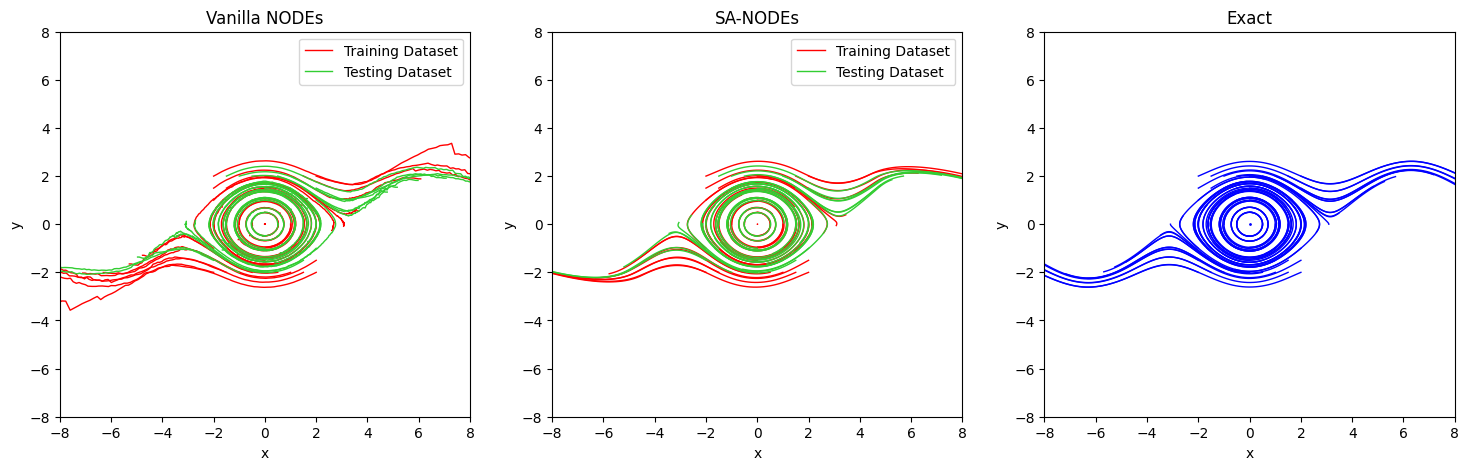}
\caption{Vanilla NODEs, SA-NODEs and exact solution.}
\label{fig:com_autoODE_sol}
\end{subfigure}
\begin{subfigure}[t]{0.27\textwidth}
\includegraphics[width=\linewidth]{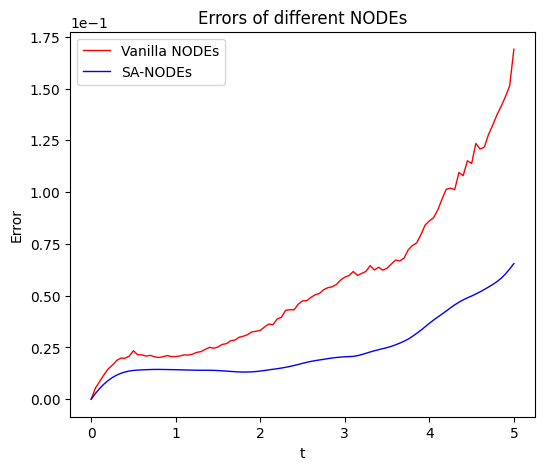}
\caption{Testing errors.}
\label{fig:com_linearODE_error}
\end{subfigure}
\caption{Comparison of vanilla NODEs and SA-NODEs on solutions and errors for system \eqref{eq:exam_ODE_auto}.}
\label{fig:com_autoODE}
\end{figure}

\begin{figure}[h!]
\centering
\begin{subfigure}[t]{0.72\textwidth}
\includegraphics[width=\linewidth]{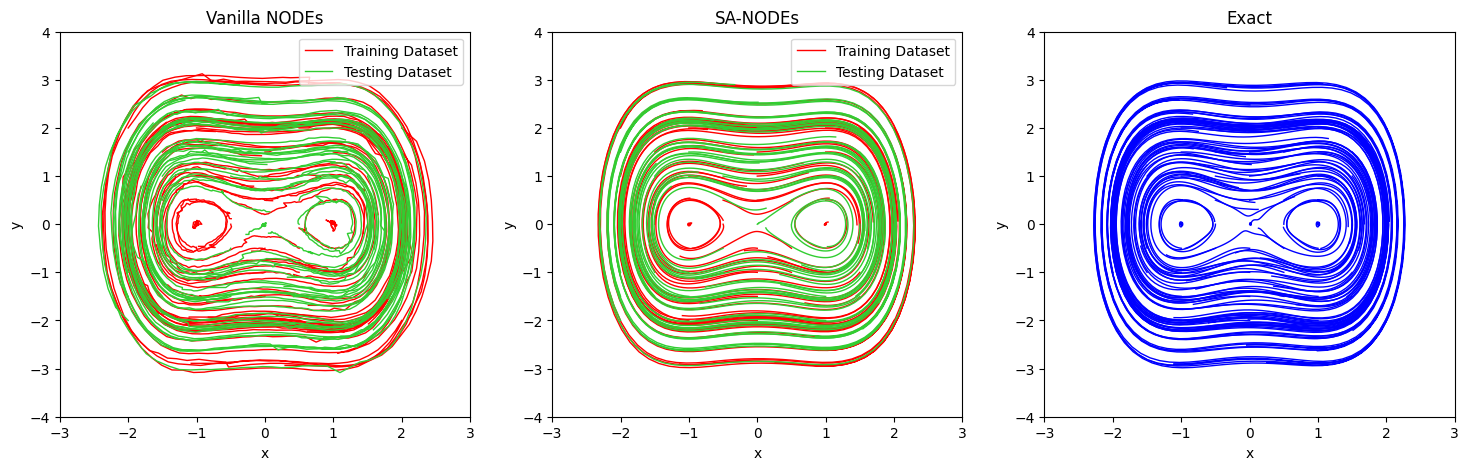}
\caption{Vanilla NODEs, SA-NODEs and exact solution.}
\label{fig:com_nonautoODE_sol}
\end{subfigure}
\begin{subfigure}[t]{0.27\textwidth}
\includegraphics[width=\linewidth]{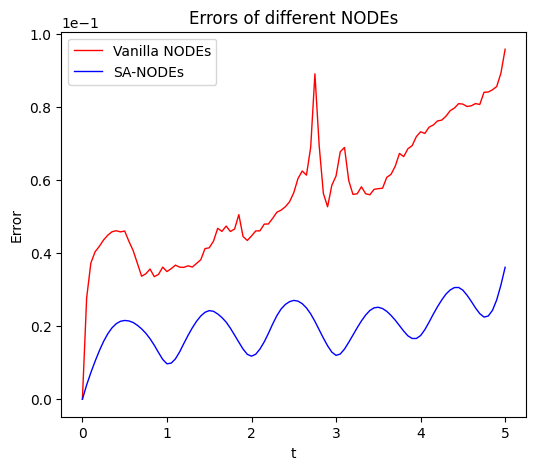}
\caption{Testing errors.}
\label{fig:com_nonautoODE_error}
\end{subfigure}
\caption{Comparison of vanilla NODEs and SA-NODEs on solutions and errors for system \eqref{eq:exam_ODE_nonauto}.}
\label{fig:com_nonautoODE}
\end{figure}

\begin{table}[h!]
\begin{center}
\resizebox{\textwidth}{!}{%
\begin{tabular}{|c|c|c|c|c|c|c|c|}
\hline
\multirow{2}{*}{$P$} & \multirow{2}{*}{Neural ODEs} & \multicolumn{3}{c|}{Autonomous Case} & \multicolumn{3}{c|}{Non-Autonomous Case} \\ \cline{3-8} 
& & \multicolumn{1}{c|}{$e_{\text{max}}$} & \multicolumn{1}{c|}{$e_{T}$} & DoF & \multicolumn{1}{c|}{$e_{\text{max}}$} & \multicolumn{1}{c|}{$e_{T}$} & DoF \\ \hline
\multirow{2}{*}{100} & Vanilla NODEs & 1.88e-01 & 1.88e-01 & 1e+06  & 1.17e+00 & 9.93e-02 & 1e+06\\ \cline{2-8} 
& SA-NODEs & 9.78e-02 & 9.78e-02 & 1e+03 & 5.46e-02 & 5.46e-02 & 1e+03 \\ \hline
\multirow{2}{*}{500} & Vanilla NODEs & 1.69e-01 & 1.69e-01 & 5e+06  & 9.62e-02 & 9.62e-02 & 5e+06  \\ \cline{2-8} 
& SA-NODEs & 8.97e-02 & 8.97e-02 & 5e+03 & 3.61e-02 & 3.61e-02 & 5e+03 \\ \hline
\multirow{2}{*}{1000} & Vanilla NODEs & 1.52e-01 & 1.52e-01 & 1e+07 & 9.57e-02 & 9.57e-02 & 1e+07 \\ \cline{2-8} 
& SA-NODEs & 6.55e-02 & 6.55e-02 & 1e+04 & 3.44e-02 & 3.44e-02 & 1e+04 \\ \hline
\end{tabular}%
}
\caption{Comparison of errors and degrees of freedom (DoF) between vanilla NODEs and SA-NODEs on autonomous and non-autonomous ODEs.}
\label{table:comp}
\end{center}
\end{table}

\begin{figure}
\centering
\begin{subfigure}[t]{0.49\textwidth}
\includegraphics[width=\linewidth]{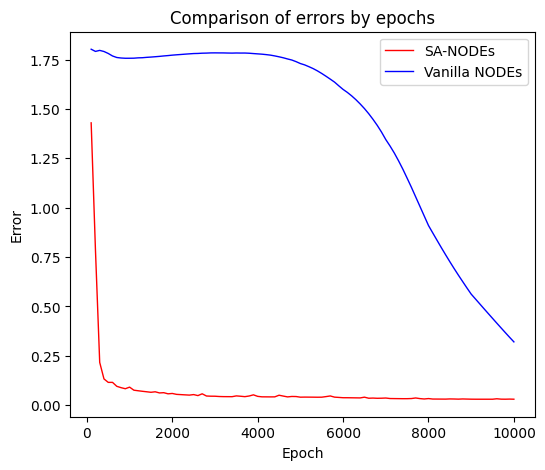}
\label{fig:com_nonautoODE_epoch}
\end{subfigure}
\begin{subfigure}[t]{0.50\textwidth}
\includegraphics[width=\linewidth]{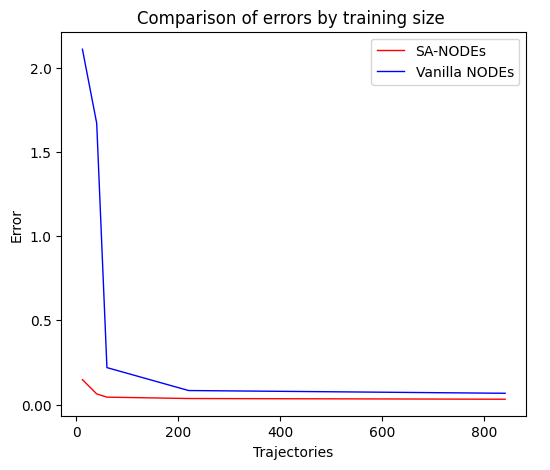}
\label{fig:com_nonautoODE_data}
\end{subfigure}
\caption{Comparison of test errors for vanilla NODEs and SA-NODEs on system \eqref{eq:exam_ODE_nonauto}: 
(Left) Training set size fixed at 220 trajectories, number of training epochs varies from $10$ to $10^4$; 
(Right) Number of training epochs fixed at $10^4$, training set size varies from 12 to 840 trajectories.}
\label{fig:com_nonautoODE_epoch_data}
\end{figure}

%%%%%%%%%%%%%%%%%%%%%%%%%%%%%%%%%%%%
\subsection{Simulations of Transport Equations}
\label{sec:neural_trans} In this subsection, we apply SA-NODEs to simulate the solutions of transport equations, thereby demonstrating their approximation performance as investigated in Theorem \ref{th:Trans}. We begin with a toy example of a non-autonomous transport equation to illustrate the training strategy mentioned in Remark \ref{remark:trans}. Using the same method, we then examine the approximation performance on an example of Doswell frontogenesis \cite{doswell1984kinematic}. The training approach relies on reformulating the transport equation as its corresponding characteristic ODE system (see Remark~\ref{remark:trans}), which requires computing derivatives of the activation function.
In this context, we use the Sigmoid activation function instead of ReLU to ensure differentiability.
Thanks to Remark~\ref{rem:sigma}, the Sigmoid-based SA-NODE achieves universal approximation performance comparable to that of the ReLU-based version in the aggregate sense. In the following experiments, we set the number of neurons in each layer to $P=200$. The learning rate is initialized at $10^{-3}$ and adjusted by a scheduler, reducing it by a factor of 0.8 every $10000$ epochs, for a total of $50000$ training epochs.

\bigskip
\noindent
\textbf{Example 3: Non-Autonomous Transport Equation}
\medskip

\noindent We focus on the following two-dimensional non-autonomous transport equation:
\begin{equation}
\label{eq:trans_exam}
\begin{dcases}
    \partial_t \rho(x,y,t) + {\rm{div}} \left(\left(
    \frac{\sin (x)}{1+t^2} , \frac{\sin (y)}{1+t^2}\right)\rho(x,y,t) \right)= 0, \quad (x,y, t)\in \R^2\times  [0,T],\\
    \rho(\cdot,0) = \rho_0 \in \mathcal{M}(\R^2).
\end{dcases}
\end{equation}
Thanks to Remark \ref{remark:trans}, it is sufficient to approximate the following characteristic system of \eqref{eq:trans_exam}:
\begin{equation*}
    \begin{dcases}
    \frac{dx}{dt} = \frac{\sin(x)}{1+t^2}, \\
    \frac{dy}{dt} = \frac{\sin(y)}{1+t^2}, \\
    \frac{d\rho}{dt} = -\rho \cdot \frac{\cos(x)+\cos(y)}{1+t^2},
\end{dcases}
\end{equation*}
where $t\in [0,T]$. We train the SA-NODE using samples drawn from a simple initial distribution: a uniform measure on the set \(K = [-4,4]^2\),
\begin{equation}
\label{eq:trans_init_train}
    \rho_0^{\text{train}}(x,y)= 0.5,  \quad (x,y)\in [-4,4]^2.
\end{equation}
On the other hand, to evaluate the performance of the trained model, we adopt a different initial distribution for testing: a truncated Gaussian measure on \(K\),
\begin{equation}
\label{eq:trans_init_test}
    \rho_0^{\text{test}}(x,y)=e^{-\frac{x^2+y^2}{4}},   \quad (x,y)\in [-4,4]^2.
\end{equation}
Let \(\rho_{\Theta}\) and \(\rho\) be solutions of the neural and the true transport equation, respectively, both initialized with the same data measure \eqref{eq:trans_init_test}. To quantify the approximation performance of $\rho_{\Theta}$, we define the following normalized testing error for each time step \( t \in [0,5] \):

\begin{equation*}
    e_{\text{test}} (t) = \|\bar{\rho}_{\Theta}(\cdot,t) - \bar{\rho}(\cdot, t)\|_{\mathbb{L}^1}, \; \text{where }\bar{\rho}_{\Theta}(\cdot,t)= \frac{\rho_{\Theta}(\cdot,t)}{\|\rho(\cdot,0)\|_{\mathbb{L}^1}} \; \text{and}\; \bar{\rho}(\cdot,t) =\frac{\rho(\cdot,t)}{\|\rho(\cdot,0)\|_{\mathbb{L}^1}}.
\end{equation*}
Solutions $\rho_{\Theta}$ and $\rho$ share the same normalization factor because of the positivity and identical form of the initial measure, along with the mass conservation property of the transport equation.
Here, we measure errors in the $\mathbb{L}^1$ norm rather than the Wasserstein-1 distance $\mathbb{W}_1$ (as in Theorem~\ref{th:Trans}), for the following reasons:
\begin{enumerate}
  \item The initial distributions are absolutely continuous and compactly supported, so the solutions remain in $\mathbb{L}^1(\R^2)$ at all times. Computing the $\mathbb{L}^1$ distance is substantially simpler than evaluating $\mathbb{W}_1$, which entails solving a numerical optimal transport problem.
  \item The $\mathbb{W}_1$ error can be bounded by the $\mathbb{L}^1$ error via
  \[
   \mathbb{W}_1\bigl(\bar{\rho}_{\Theta}(\cdot,t),\,\bar{\rho}(\cdot,t)\bigr)
    \;\le\;
    \frac{\operatorname{diam}(\Omega_t)}{2}\,
    \bigl\lVert \bar{\rho}_{\Theta}(\cdot,t)-\bar{\rho}(\cdot,t)\bigr\rVert_{\mathbb{L}^1},
  \]
  where $\mathrm{diam}(\Omega_t)$ denotes the diameter of the common support of $\bar{\rho}_{\Theta}(\cdot,t)$ and $\bar{\rho}(\cdot,t)$, which remains finite by Grönwall’s lemma.
\end{enumerate}

\medskip
For the training dataset, initial locations are sampled on the grid $[-4,4]^2$ with spacing 0.2 (1681 trajectories). For the testing dataset, to assess generalization over the state space, we use a denser grid $[-4,4]^2$ with spacing $0.1$, yielding $6561$ initial conditions and corresponding trajectories for testing.

In Figure~\ref{fig:trans2D}, we display, from top to bottom, the solution obtained by the SA-NODE, the solution obtained by the vanilla NODE, and the exact solution of the transport equation on the domain \((x,y) \in [-4,4]^2\) at 51 equispaced time points \(t \in [0,5]\) (including the two extrema 0 and $T$). Figure~\ref{fig:trans_2D_exam1} presents the corresponding approximation errors: the left panel shows the training and testing errors of the SA-NODE, while the right panel compares the testing performance of the SA-NODE and the vanilla NODE.

From Figure~\ref{fig:trans2D}, we observe that both neural models provide good approximations of the true dynamics. In Figure~\ref{fig:trans_2D_exam1}, the testing errors remain consistently low, on the order of less than \(10^{-1}\). The right panel clearly shows that the vanilla NODE performs worse than the SA-NODE in the early stages, though both models eventually converge to similar accuracy. This highlights the stability advantage of the SA-NODE. Moreover, since the theoretical error (see Theorems~\ref{th:rates} and~\ref{th:Trans}) is defined as the maximum over time, the SA-NODE yields better approximation performance in this robustness sense.

\begin{figure}
     \centering
	\includegraphics[width=\textwidth]{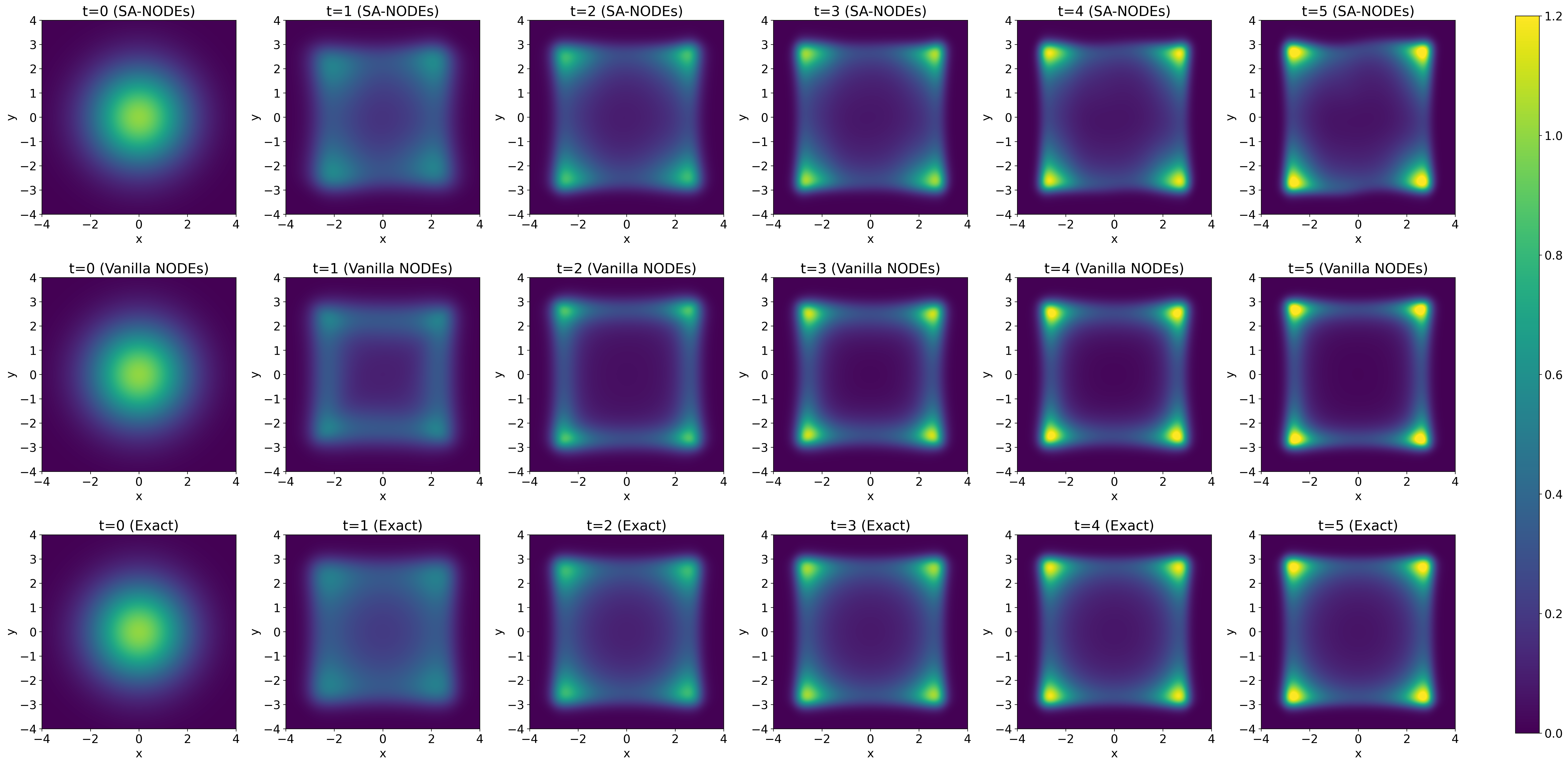}
     \caption{SA-NODEs, vanilla NODEs and exact solutions of transport equation \eqref{eq:trans_exam} with initial measure \eqref{eq:trans_init_test}.}
     \label{fig:trans2D}
\end{figure}

\begin{figure}
\centering
\begin{subfigure}[t]{0.48\textwidth}
\includegraphics[width=\linewidth]{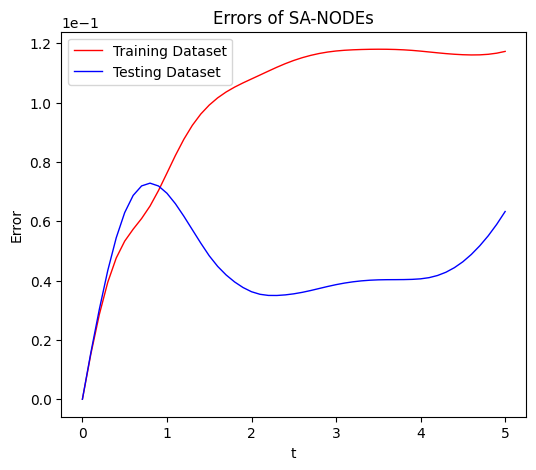}
\label{fig:trans_2D_exam1_SA}
\end{subfigure}
\begin{subfigure}[t]{0.49\textwidth}
\includegraphics[width=\linewidth]{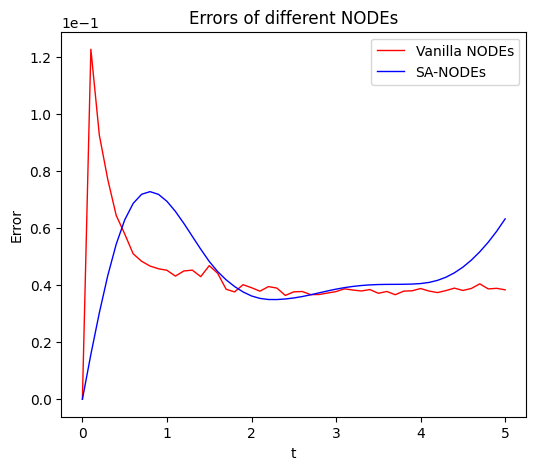}
\label{fig:trans_2D_exam1_compare}
\end{subfigure}
\caption{Training and testing errors of SA-NODEs and comparison with vanilla NODEs on testing errors for transport equation \eqref{eq:trans_exam}.} \label{fig:trans_2D_exam1}
\end{figure}

\bigskip
\noindent
\textbf{Example 4: Doswell Frontogenesis}
\medskip

\noindent We now consider the two-dimensional Doswell frontogenesis equation \cite{doswell1984kinematic, morales2019adjoint}. This model describes the insurgence and evolution of horizontal temperature gradients and fronts within meteorological dynamics. The equation reads:
\begin{equation}
\label{eq:trans_Doswell}
\begin{dcases}
    \partial_t \rho(x,y,t) + {\rm{div}} \left(\left(
    -yg(r(x,y)) , xg(r(x,y))\right)\rho(x,y,t) \right)= 0, \quad \text{in }\, \R^2\times  [0,T],\\
    \rho(\cdot,0) = \rho_0,
\end{dcases}
\end{equation}
where
\begin{equation} 
    g(r(x,y))=\frac{1}{r(x,y)}\ \overline{v}\ {\text{sech}}^2{(r(x,y))} \tanh{(r(x,y))},
\end{equation}
with $r(x,y)=\sqrt{x^2+y^2}$ and $\overline{v}=2.59807$. 
The initial measures for the training and testing are set as: 
\begin{equation*}
    \rho_0^{\text{train}}(x,y)=\tanh \left(y\right), \quad \rho_0^{\text{test}}(x,y)=\tanh \left(10 \, y\right).
\end{equation*}
With this choice of initial data, the exact solution of \eqref{eq:trans_Doswell} can actually be computed by hand:
\[
u(x,y,t) = \tanh\left(\frac{y \cos(g(r)t) - x\sin(g(r)t)}{\delta}\right),
\]
where $\delta = 1$ for $\rho_0^{\text{train}}$ and $\delta = 1/10$ for $\rho_0^{\text{test}}$.

To generate the training and testing datasets, we define the domain $K=[-5,5]^2$ and sample initial conditions on regular grids with spacing $0.2$ (yielding $2601$ trajectories) for training, and spacing $0.1$ (yielding $10201$ trajectories) for testing. The time discretization is the same as in the previous experiment.

Figure~\ref{fig:trans2D_Doswell} displays the SA-NODE solution, the vanilla NODE solution, and the exact solution corresponding to the testing initial measure. A near-perfect alignment is observed across the entire time horizon \([0,4]\). The error curves (defined as in Example~3) are presented in Figure~\ref{fig:trans_2D_exam2}, showing a consistently low error level on the order of \(10^{-3}\). Once again, the right panel highlights the instability of the vanilla NODE, where a noticeable spike in error occurs around time \(t = 3.7\). This further demonstrates that the SA-NODE approximation has better robustness and stability.

\begin{figure}
     \centering
	\includegraphics[width=\textwidth]{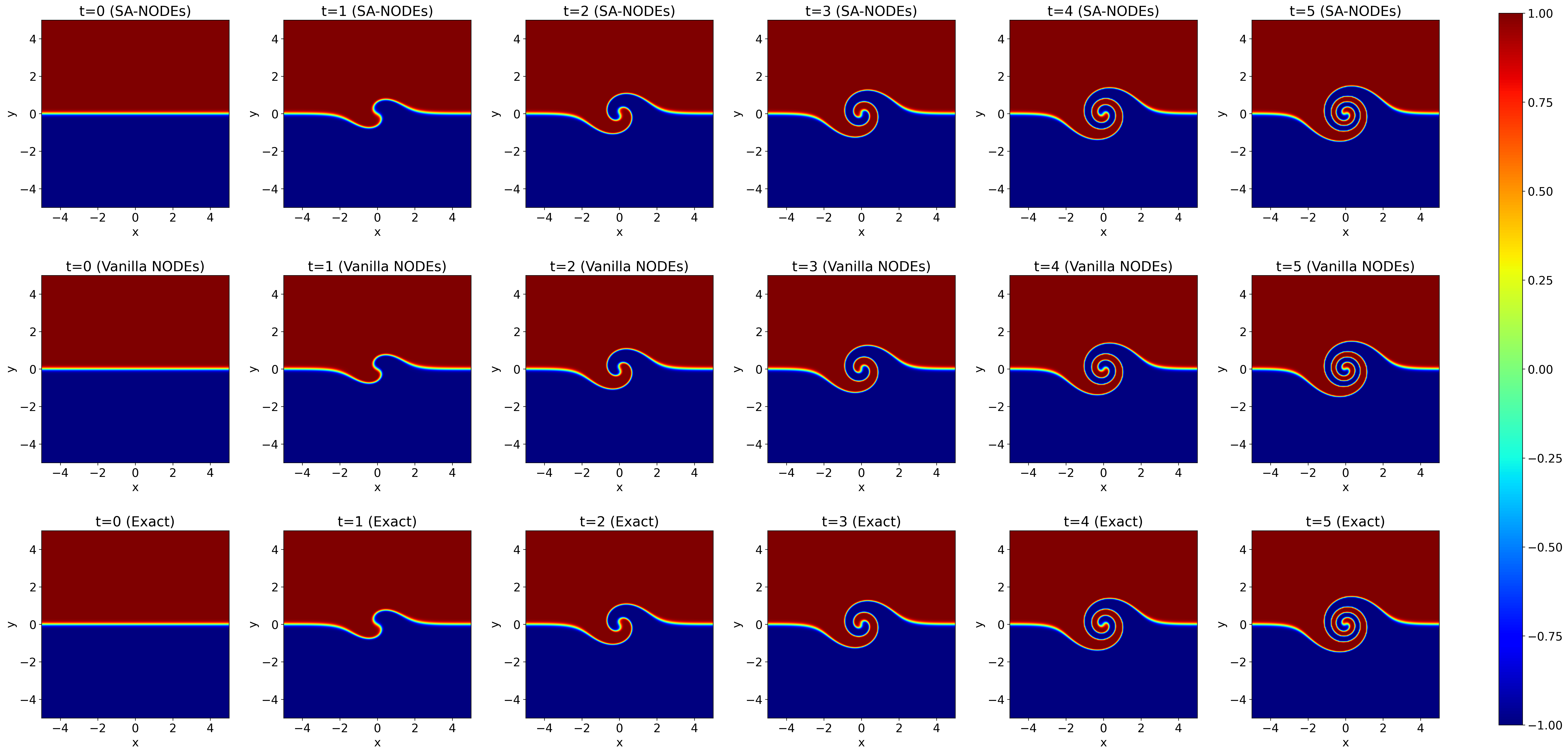}
     \caption{SA-NODEs, vanilla NODEs and exact solutions of transport equation \eqref{eq:trans_Doswell} with the testing initial measure.}
     \label{fig:trans2D_Doswell}
\end{figure}

\begin{figure}
\centering
\begin{subfigure}[t]{0.50\textwidth}
\includegraphics[width=\linewidth]{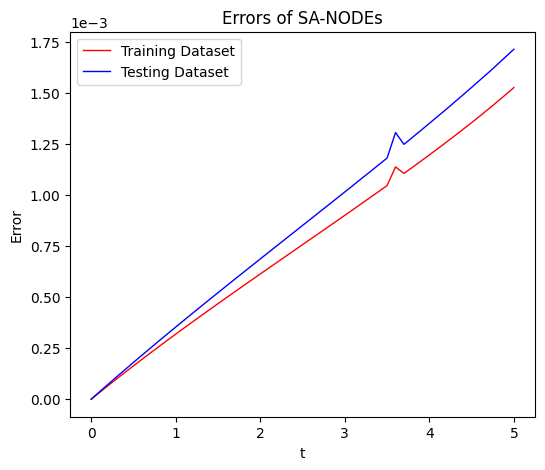}
\label{fig:trans_2D_exam2_SA}
\end{subfigure}
\begin{subfigure}[t]{0.48\textwidth}
\includegraphics[width=\linewidth]{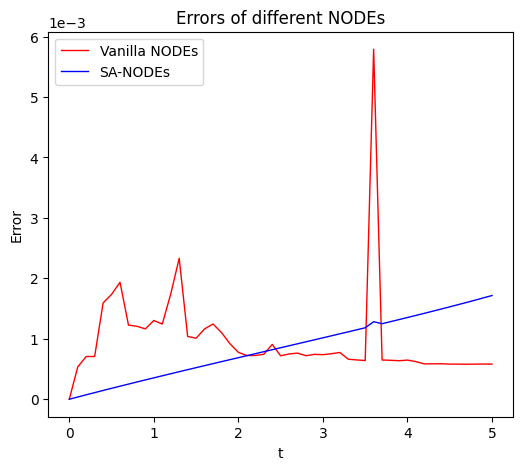}
\label{fig:trans_2D_exam2_compare}
\end{subfigure}
\caption{Training and testing errors of SA-NODEs and comparison with vanilla NODEs on testing errors for transport equation \eqref{eq:trans_Doswell}.} \label{fig:trans_2D_exam2}
\end{figure}

\section{Conclusions and future works}\label{sec:conclusion}

\noindent
In this paper, we have introduced SA-NODEs, a novel framework for modeling and approximating dynamical systems. Our theoretical analysis establishes the universal approximation properties and convergence rate of SA-NODEs, demonstrating their ability to approximate dynamical systems. We have highlighted that training SA-NODEs is akin to solving an optimal control problem, where the objective is to reconstruct the underlying dynamical system.

The numerical experiments validate the effectiveness of SA-NODEs across various scenarios, including linear and nonlinear ODE systems and transport equations. The results show that SA-NODEs consistently outperformed vanilla NODEs in terms of accuracy and computational efficiency. This superior performance is attributed to the reduced complexity of SA-NODEs, which require fewer parameters and training epochs compared to their vanilla counterparts. Furthermore, SA-NODEs exhibited robust generalization capabilities, maintaining low error rates even with limited training datasets. 

The novelty of the SA-NODE framework opens up several possibilities for future investigation. 
A first research direction may focus on improving the results obtained in this work in the case of specific dynamical systems, for example, gradient systems (e.g.\@ Hamiltonian system), equations exhibiting periodical dynamics, autonomous systems, etc. In other words, it would be interesting to study to what extent SA-NODEs are able to capture distinct properties of the dynamical system generating the data, and whether it is possible to achieve better approximation results in specific situations. For instance, in the Hamiltonian setting, results from a recent work \cite{gonon2023approximation} can be applied to achieve a more precise approximation in the probabilistic sense. 
Besides, in the autonomous case, as mentioned in Remark \ref{rem:autonomous}, the SA-NODE also becomes autonomous. Consequently, further studies on the relation between the approximation quality and stability of both the original and neural systems can be conducted.

A second path of exploration involves the predictive properties of SA-NODEs. Indeed, since the coefficients are fixed in time, it is theoretically possible to solve the SA-NODE for times that exceed the time $T$ up to which data were available, effectively predicting the dynamics. This property is exclusive to SA-NODEs, and studying to which extent these equations are able to stay close to the real dynamics after time $T$ is a very enticing question. This prediction task is closely related to the well-known recurrent neural network (RNN) for time series. A particular type of RNN, known as echo state networks (ESNs), prohibits UAP for discrete dynamical systems in the infinite time horizon, as demonstrated in \cite{grigoryeva2018echo}.  Nevertheless, infinite-horizon estimates might be obtained in some cases for SA-NODE, by relying on the theory of Lyapunov Exponents, much in the same spirit of \cite{berry2023learning, grigoryeva2024forecasting}. In future work, we can adapt ESNs to the continuous-time scenario and compare their prediction performances with those of the SA-NODE.

In this work, we primarily focus on ODE systems and associated transport equations. However, the applicability of SA-NODEs extends beyond this setting: they can also be employed to interpolate data from more general dynamical systems that are not necessarily governed by ODEs. In this way, our approach reduces the complexity of capturing the main components of a complex dynamical system within an ODE-based framework. For example, SA-NODEs can be used to reconstruct an underlying deterministic ODE model from data produced by a randomly perturbed version of the system. Looking forward, this perspective opens several promising directions, including the development of stochastic or hybrid extensions of SA-NODEs capable of handling richer classes of dynamical behaviors, and the exploration of their role as interpretable, data-driven models in scientific machine learning.

\bibliographystyle{abbrv}
\bibliography{references}
\end{document}